\documentclass[12pt]{amsart}
\usepackage{a4wide}
\usepackage[utf8]{inputenc}
\usepackage{amsmath, amssymb, bbm}
\usepackage{amsthm}
\usepackage{mathtools}
\usepackage{bbm}
\usepackage{blkarray}
\usepackage[matrix,arrow,curve]{xy}
\usepackage{tikz}
\usepackage{tikz-cd}
\usepackage{graphicx}
\usepackage{color}
\usepackage[normalem]{ulem}
\usepackage{enumerate}
\usepackage{enumitem}
\definecolor{darkblue}{rgb}{0,0,0.6}
\usepackage[ocgcolorlinks,colorlinks=true, citecolor=darkblue, filecolor=darkblue, linkcolor=darkblue, urlcolor=darkblue]{hyperref}
\usepackage[capitalize,noabbrev]{cleveref}
\usepackage{comment}
\usepackage[mathscr]{euscript}
\usepackage{stmaryrd}

\usepackage[draft,margin,author=UP,silent]{fixme}

\newcounter{commentcounter}

\let\oldtocsection=\tocsection

\let\oldtocsubsection=\tocsubsection

\let\oldtocsubsubsection=\tocsubsubsection

\renewcommand{\tocsection}[2]{\hspace{0em}\oldtocsection{#1}{#2}}
\renewcommand{\tocsubsection}[2]{\hspace{1em}\oldtocsubsection{#1}{#2}}
\renewcommand{\tocsubsubsection}[2]{\hspace{2em}\oldtocsubsubsection{#1}{#2}}

\newtheorem{theorem}{Theorem}[section]
\newtheorem{lem}[theorem]{Lemma}
\newtheorem{cor}[theorem]{Corollary}
\newtheorem{prop}[theorem]{Proposition}

\theoremstyle{definition}

\newtheorem{definition}[theorem]{Definition}
\newtheorem{ex}[theorem]{Example}
\newtheorem{rem}[theorem]{Remark}

\def\hB{\hspace*{\fill}$\qed$}
\DeclareMathOperator{\colim}{colim}

\newcommand{\st}{\mathrm{st}}
\newcommand{\Mon}{\mathbf{Mon}}
\newcommand{\Rex}{\mathrm{Rex}}
\newcommand{\CH}{\mathbf{CH}}

\newcommand{\alg}{\mathrm{alg}}

\newcommand{\Simp}{\mathbf{Simp}}

\newcommand{\sepa}{\mathrm{sep}}

\newcommand{\fin}{\mathrm{fin}}

\newcommand{\group}{\mathrm{group}}

\newcommand{\nCalg}{C^{*}\mathbf{Alg}
}

\newcommand{\F}{\mathbb{F}}

\newcommand{\ho}{\mathrm{ho}}

\newcommand{\Nat}{\mathrm{Nat}}

\newcommand{\CMon}{\mathbf{CMon}}

\newcommand{\bB}{{\mathbf{B}}}
\newcommand{\Fib}{{\mathrm{Fib}}}

\newcommand{\incl}{\mathrm{incl}}

\newcommand{\GL}{\mathrm{GL}}

\newcommand{\bM}{\mathbf{M}}

\newcommand{\bF}{{\mathbf{F}}}
\newcommand{\CAlg}{{\mathbf{CAlg}}}

\newcommand{\Ext}{{\mathrm{Ext}}}

\newcommand{\bA}{{\mathbf{A}}}

\newcommand{\const}{{\mathtt{const}}}

\newcommand{\Alg}{{\mathbf{Alg}}}

\newcommand{\cO}{{\mathcal{O}}}

 \newcommand{\Cat}{{\mathbf{Cat}}}

\newcommand{\Group}{\mathbf{Group}}

\newcommand{\Spc}{\mathbf{Spc}}

\newcommand{\IN}{\mathbb{N}}

\newcommand{\Calg}{{\mathbf{C}^{\ast}\mathbf{Alg}}}

\newcommand{\id}{\mathrm{id}}

\newcommand{\markus}[1]{\textcolor{blue}{#1}}

\newcommand{\ulrich}[1]{\textcolor{teal}{#1}}

\newcommand{\inv}{\mathrm{inv}}

\newcommand{\bC}{\mathbf{C}}
\newcommand{\beins}{\mathbf{1}}

\newcommand{\kk}{\mathrm{kk}}
\newcommand{\K}{\mathrm{K}}
\newcommand{\KK}{\mathbf{KK}}
\newcommand{\KO}{\mathrm{KO}}
\newcommand{\KU}{\mathrm{KU}}

\renewcommand{\Group}{\mathbf{Group}}

\newcommand{\E}{\mathbb{E}}
\renewcommand{\beins}{\mathbbm{1}}

\renewcommand{\O}{\mathscr{O}}

\newcommand{\UCT}{\mathrm{UCT}}
\newcommand{\Ab}{\mathbf{Ab}}
\newcommand{\Z}{\mathbb{Z}}
\newcommand{\nat}{\IN}
\newcommand{\Hom}{\mathrm{Hom}}
\newcommand{\cT}{\mathcal{T}}
\newcommand{\Mod}{\mathbf{Mod}}
\newcommand{\cS}{\mathcal{S}}
\newcommand{\R}{\mathbb{R}}
\renewcommand{\P}{\mathbb{P}}
\newcommand{\Sp}{\mathbf{Sp}}
\newcommand{\Top}{\mathbf{Top}}
\newcommand{\cC}{\mathcal{C}}
\newcommand{\op}{\mathrm{op}}
\newcommand{\map}{\mathrm{map}}
\newcommand{\Map}{\mathrm{Map}}
\newcommand{\Tor}{\mathrm{Tor}}
\newcommand{\End}{\mathrm{End}}
\newcommand{\bT}{\mathscr{T}}
\newcommand{\Fun}{\mathbf{Fun}}
\newcommand{\Ind}{\mathrm{Ind}}
\newcommand{\cA}{\mathcal{A}}
\newcommand{\Aut}{\mathrm{Aut}}

\newcommand{\lto}{\longrightarrow}

\newcommand{\C}{\mathbb{C}}
\newcommand{\bD}{\mathbf{D}}
\newcommand{\bS}{\mathbb{S}}

\title[A survey on $K$-theory via homotopical algebra]{A survey on operator $K$-theory \\ via homotopical algebra}
\date{\today}

\thanks{Ulrich Bunke was supported by the SFB 1085 (Higher Invariants) funded by the Deutsche Forschungsgemeinschaft (DFG)}

\author[U.~Bunke]{Ulrich Bunke}
\address{Universit\"at Regensburg, Mathematisches Institut, Universit\"atsstr. 31, 93053 Regensburg, Germany}
\email{ulrich.bunke@mathematik.uni-regensburg.de}

\author[M.~Land]{Markus Land}
\address{Mathematisches Institut, Ludwig-Maximilians-Universit\"at M\"unchen, Theresienstra\ss e 39, 80333 M\"unchen, Germany}
\email{markus.land@math.lmu.de}

\author[U.~Pennig]{Ulrich Pennig}
\address{Cardiff University, School of Mathematics, Senghennydd Road, Cardiff, CF24 4AG, Wales, UK} 
\email{PennigU@cardiff.ac.uk}

\begin{document}
\bibliographystyle{alpha}

\begin{abstract}
This is a survey article with the goal to advertise spectrum valued versions of $K$- and $KK$- theory for $C^{*}$-algebras via a (stable and symmetric monoidal) $\infty$-categorical enhancement of Kasparov's classical $KK$-theory. The main purpose is to present, in the simplest case, homotopy theoretic arguments for classical results on operator $K$-theory, including Swan's theorems, K\"unneth and universal coefficient formulas, the bootstrap class, variations of Karoubi's conjecture, and spectra of units for strongly self-absorbing $C^*$-algebras, as well as some new aspects on twisted $K$-theory and coherent multiplicative structures on $C^*$-algebras, viewed as objects in the previously mentioned $\infty$-category. 
\end{abstract}

\maketitle

\tableofcontents

\section{Introduction}
$KK$-theory was introduced originally by Kasparov in his work about the Novikov conjecture \cite{kasparovinvent}. It is a bivariant generalisation of operator-algebraic $K$-theory which combines it with its dual, called $K$-homology theory. As evident from its definition via Fredholm modules $KK$-theory provides a natural receptacle for index invariants. 
It has applications in representation theory and is one of the most important invariants in the classification programme of simple, nuclear $C^*$-algebras. Kasparov's abelian group valued functor $(A,B) \mapsto KK(A,B)$ is contravariant in $A$ and covariant in $B$. Kasparov furthermore defined a unital, bilinear and associative intersection product
\[
	KK(B,C) \otimes KK(A,B) \to KK(A,C)\ . 
\]
This allows to interpret the groups $KK(A,B)$ as the hom-groups of an additive category $\mathrm{KK}$ whose objects are the $C^{*}$-algebras, and in which the intersection product is the composition of morphisms. Operator-algebraic $K$-theory and $K$-homology can be recovered from the $KK$-groups via natural isomorphisms $K_0(A) \cong KK(\C,A)$ and $K^0(A) \cong KK(A,\C)$. In his work on Karoubi's conjecture on the algebraic $K$-theory of stable $C^*$-algebras, Higson characterized the additive category $\mathrm{KK}$ by a universal property \cite{Higson87, higsa}. His work already suggested that $\mathrm{KK}$ should be a (1-categorical) localization of the category of separable $C^*$-algebras which would provide a universal characterisation of the category $\mathrm{KK}$ without appealing to its additivity.
In his work on homotopy theory for $C^*$-algebras \cite{Uuye:2010aa}, Uuye showed that this is indeed the case. Higson's characterization remains, however, the most useful one in practice.

Some 20 years after Higson's work, Meyer and Nest proved in \cite{MR2193334} that $\mathrm{KK}$ has in fact a natural triangulated structure. With Lurie's development of algebra in the context of higher categories \cite{HA}, it was then natural to ask whether there exists a stable $\infty$-category such that its underlying homotopy category reproduces this triangulated category (note that this happens for example in the case of the stable homotopy category). {That this is indeed the case}
was established in work of the second author with Nikolaus \cite{LN}, and used to obtain spectrum valued comparisons between the topological $K$-theory and the algebraic $L$-theory of complex $C^*$-algebras. The motivation behind such a comparison was to provide an equivalence between the $L$-theoretic Farrell--Jones conjecture and the Baum--Connes conjecture away from 2. Moreover, in \cite{LNS} similar methods were used to determine the algebraic $L$-spectra and $L$-groups of real $C^*$-algebras in terms of the topological $K$-spectrum and $K$-groups. In \cite{KKG}, the case of $G$-equivariant $KK$-theory for $G$ a discrete group was worked out and used in \cite{bel-paschke} to formulate and prove an equivariant form of Paschke duality.

In this survey we will explain how the use of (Dwyer--Kan) localizations in the world of $\infty$-categories leads to the stable $\infty$-category $\KK$ which enhances the triangulated category $\mathrm{KK}$ of Meyer--Nest, following \cite{LN}, and several of its applications. 

A consequence of this construction is the realisation of $K$- and $KK$-theory groups as the homotopy groups of corresponding spectra. 
Most of the literature about $K$-theory is only concerned with $K$- and $KK$-theory groups, and the spectrum aspect is considered more like a homotopy theoretic curiosity,
except for the above mentioned papers, some applications to assembly maps 
\cite{joachimcat} \cite{joachim} and \cite{DP1,DP2,DP3} in which bundles of stabilised strongly self-absorbing $C^*$-algebras are classified via the unit spectrum of $K$-theory and its localisations. We think that the homotopy-theoretic picture was neglected, because constructions of $K$- and $KK$-theory spectra using point-set models are technically involved, in particular when including multiplicative structures.

The goal of the present paper is to advertise the spectrum-valued point of view on $K$- and $KK$-theory, 
by providing examples where working with spectra leads to simplifications of statements and calculations, in particular, when certain calculations have already been done by topologists. The present article focuses on the ordinary theory of complex $C^*$-algebras, but analogs of our approach exist (and have already been exploited) in various generalizations, e.g.\ for real $C^*$-algebras and $G$-equivariant $C^*$-algebras. $E$-theory (again also in the real and equivariant cases) allows for a completely parallel treatment \cite{Bunke:2023aa}. 

Many of the homotopy theoretic formulations we present in the case of non-equivariant complex $C^*$-algebras have analogs in these more general settings. This applies e.g.\ to the \ K\"unneth and universal coefficient theorems. In the general case these results have very natural and simple formulations as statement about spectra, but their translations to groups becomes complicated. In this respect the case of complex $C^{*}$-algebras with no group action is misleadingly simple  because of the simpleness of the homological algebra  over the graded ring $K_{*}(\C)\cong \Z[\beta,\beta^{-1}]$ (it has graded global dimension $1$).

To highlight the formal simplicity of spectrum-valued $K$- and $KK$-theory  using homotopy theory, we briefly outline the situation as follows: 
{One defines the $\infty$-category $\KK_\sepa$ to be the $\infty$-categorical localization of $\nCalg_\sepa$ at the $\mathrm{KK}$-equivalences, i.e.\ all morphisms which are sent to isomorphisms in the ordinary category $\mathrm{KK}$. One then shows that $\KK_\sepa$ is stable and the localization functor $\kk_\sepa\colon \nCalg_\sepa \to \KK_\sepa$ is canonically symmetric monoidal and the initial functor to a stable $\infty$-category which is homotopy invariant, stable, and semi-exact.\footnote{We explain these properties in more detail in the body of the text. Also, homotopy invariance can in fact be dropped here.} This is the analog of Higson's characterization of the additive category $\mathrm{KK}$.\footnote{Note that a similar such characterization is not possible for the triangulated category, as it is not a property of a functor to be triangulated, rather it requires to first specify extra structure.} We then define $\KK = \Ind(\KK_\sepa)$ to be the ind-completion of $\KK_\sepa$ and obtain a canonical functor $\kk\colon \nCalg \to \KK$ defined on all $C^*$-algebras. It has similar (universal) properties as $\kk_{\sepa}$.}

 The composition of the functor \begin{equation}\label{rfekrjvkofbvfsdvfvsfdv}\map_{\KK}( \beins,-)\colon \KK\to \Sp   
\end{equation} with $\kk$
turns out to be equivalent to a spectrum valued enhancement of the classical operator $K$-theory functor {due to Joachim \cite{joachimcat, joachim}}. {Combining the universal property of $\kk$ with appropriate versions of the Yoneda lemma, one can characterize the spectrum-valued $K$-theory functor by a universal property. In particular, it is uniquely determined by its properties.}

The applications of the homotopy theoretic point of view presented in this survey include
\begin{enumerate}
\item[-] Swans theorem in $K$- and $KK$-theory.
\item[-] K\"unneth and universal coefficient formulas, as well as a description of the UCT (aka bootstrap) class.
\item[-] Spectra of units for strongly self-absorbing $C^*$-algebras and coherent multiplicative structures.
\item[-] Variations of Karoubi's conjecture.
\item[-] $L$-theory of $C^*$-algebras.
\item[-] Twisted $K$-theory.
\end{enumerate}

The discussion of coherent multiplicative structures on $C^*$-algebras, viewed as objects of $\KK$, as well as that of twisted $K$-theory is, to the best of our knowledge, new. The other results are well-known and at most presented in new form.

We review the most important aspects of $\infty$-categories as well as $K$- and $KK$-theory for $C^*$-algebras in the body of the text, but assume that the reader is already familiar with these concepts.

\section{Higher structures on  $K$-theory}

\subsection{Classical operator $K$-theory}
We let $\nCalg$ denote the category of complex $C^{*}$-algebras, possibly non-unital, with $*$-homomorphisms which possibly do not preserve units. Classically, topological $K$-theory of $C^{*}$-algebras is considered as a functor
\begin{equation}\label{fadfdsf}
K_{*}\colon  \nCalg\to  \Ab^{\Z} 
\end{equation} 
with values in the category  $\Ab^{\Z} $ of $\Z$-graded abelian groups. We begin by summarising some of the basic constructions and properties of this functor treated in any of the standard references, e.g.\ \cite{blackadar}, \cite{higson_roe}, \cite{RordamLarsenLaustsen}.
 
\begin{enumerate}[label=(\roman*)]

\item \textbf{Definitions:}  For $A$ in $\nCalg$ the groups $K_{0}(A)$ and the higher $K$-groups $K_{i}(A)$ for $i>0$, respectively,  are  defined explicitly in terms  {of homotopy classes} of projections or homotopy groups of the space of unitaries in matrix algebras over the unitalization $A^{+}$  of $A$. In greater detail, if $A$ is unital, then 
\[K_{0}(A):=\pi_{0}(P^{s}(A))^{\group}\ ,\]
where  $$P^{s}(A):=\Hom_{\nCalg}(\C,A\otimes \mathbb{K})$$ is the space of stable projections of $A$ and $\mathbb{K}$ is the algebra of compact operators on a separable Hilbert space. Its set of  components has a semigroup structure given by the orthogonal sum of projections, and the superscript $(-)^{\group}$ indicates group completion. If $A$ is non-unital, then one sets  $$K_{0}(A):=\ker(K_{0}(A^{+})\to K_{0}(\C))\ ,$$  where $A^{+}$ is the unitalization of $A$ and $A^{+}\to \C$ is the canonical homomorphism.
For $i\ge 1$ one defines $$K_{i}(A):=\pi_{i-1}(U^{s}(A))\ ,$$ where $$U^{s}(A):=\{U\in U(A\otimes \mathbb{K})^{+}\mid 1-U\in A\otimes \mathbb{K}\}$$ is the topological group of stable unitaries of $A$. The groups $K_{i}(A)$ for $i<0$ are not  a priori defined.

\item \textbf{Bott periodicity:}
It turns out that the groups $K_0(A)$ and $K_1(A)$ determine all $K$-groups\footnote{ {This is a special  fact for complex $C^*$-algebras. In contrast,  in the case of  real $C^*$-algebras one must consider  eight $K$-groups.}} of $A$ since there are natural Bott periodicity isomorphisms $K_{ j}(A) \to K_{j+2}(A)$ for all $j$ in $\nat$. They are used to define $K_{i}(A)$ for $i<0$ such that $K_{j}(A)\cong \K_{j+2}(A)$ for all $j$ in $\Z$.

\item \textbf{Homotopy invariance:}
A homotopy between two $*$-homomorphisms $f_{0},f_{1}\colon A\to B$ is a path in $\Hom_{\nCalg}(A,B)$  which is continuous in the point-norm topology. For homotopic maps $f_0$ and $f_1$ we have the equality  $K_{*}(f_{0})=K_{*}(f_{1})$, i.e.\ the functor $K_{*}$ is  {homotopy invariant}. In particular, $K_{*}$ sends homotopy equivalences between $C^{*}$-algebras to isomorphisms.

\item \textbf{$\mathbb{K}$-stability:}
Let $p$ be a minimal non-zero projection in the algebra $\mathbb{K}$. It induces a homomorphism $\C\to \mathbb{K}$ by $\lambda\mapsto \lambda p$ called the upper corner embedding. We use the same name for the  induced map $A\cong A\otimes  \C\to A\otimes \mathbb{K}$ for any $C^{*}$-algebra $A$. The induced map $K_*(A) \to K_*(A \otimes \mathbb{K})$ is an isomorphism. 

\item \textbf{Exactness:}
Associated to a short exact sequence 
\begin{equation}\label{advasdcasdca}0\to A\stackrel{i}{\to} B\stackrel{\pi}{\to} C\to 0 \end{equation}  
of $C^{*}$-algebras, one can construct a   {boundary operator} $\partial\colon K_*(C) \to K_{*-1}(A)$
 such that the sequence \begin{equation}\label{cdscasdcasdcasdcdc}
\dots \stackrel{\partial }{\lto} K_{*}(A)\stackrel{i_*}{\lto} K_{*}(B)\stackrel{\pi_*}{\lto} K_{*}(C)\stackrel{\partial}{\lto} K_{*-1}(A) \stackrel{i_{*}}{\lto} \dots
\end{equation} 
is exact. 
The boundary operator depends naturally on the exact sequence such that a map of  short exact sequences yields a map of long-exact sequences of $K$-theory groups. 

\item \textbf{Filtered colimits:}
The functor $K_*$ preserves filtered colimits. This is shown for $K_0$ and $K_1$, using the description via projections and unitaries, and then prolongs to all $K$-groups by Bott periodicity.

\item \textbf{Exterior multiplications:}
The  category $\nCalg$ has two canonical symmetric monoidal structures given by the minimal and maximal tensor products. In order to avoid clumsy notation in the present paper we will usually work with the maximal tensor product denoted  by $\otimes$, though many assertions involving tensor products have versions for the minimal tensor product as well.\footnote{{Technically speaking, the identity refines to a lax symmetric monoidal functor $(\nCalg, \otimes_{\min}) \to (\nCalg,\otimes_{{\max}})$, so whenever we assert a statement about lax symmetric monoidal functors from $(\nCalg,\otimes_{{\max}})$ to some symmetric monoidal $(\infty)$-category, we obtain a corresponding statement for the minimal tensor product by precomposition with this lax symmetric monoidal functor.}} We equip $\Ab^{\Z}$  with the symmetric monoidal structure given by the graded tensor product. The $K$-theory functor $K_{*}$ can be complemented by exterior multiplication maps
\begin{equation}\label{afdsfadsfaewad}
K_{*}(A)\otimes K_{*}(B)\to K_{*}(A \otimes  B)
\end{equation} 
which in degree zero is induced by taking tensor products of projections. In fact, these exterior multiplication maps are part of a {lax symmetric monoidal structure} on the $K$-theory functor \eqref{fadfdsf}. 
\end{enumerate}
 
\begin{rem} \label{wtijgowergferwgwref} 
For any symmetric monoidal category $\cT$  {with tensor unit $\beins$ one can form the category $\CAlg(\cT)$ of commutative algebra objects and for $A$ in $\CAlg(\cT)$, one can further consider the category $\Mod_{\cT}(A)$ of $A$-modules in $\cT$. For example, $\beins$ refines to a commutative algebra object in $\CAlg(\cT)$ in {a} unique  way and any object of $\cT$ has  a unique  structure of a $\beins$-module, yielding a canonical equivalence}
\begin{equation}\label{afdasdfdscasdcasd}
\Mod_{\cT}(\beins) \stackrel{\simeq}{\lto} \cT\ .
\end{equation}
A lax symmetric monoidal functor $F\colon \cT\to \cS$ between symmetric monoidal categories
induces a functor  $F\colon \CAlg(\cT)\to \CAlg(\cS)$ between the categories of commutative algebras  and a lax symmetric monoidal functor 
$F\colon \Mod_{\cT}(A)\to \Mod_{\cS}(F(A))$ between corresponding module categories in a canonical way.
In particular, $F$ refines to a lax symmetric monoidal functor
\begin{equation}\label{sdvsdcsacd}
F\colon \cT\simeq \Mod_{\cT}(\beins)\to \Mod_{\cS}(F(\beins))\ .
\end{equation} 
{When $\cT$ is clear from context, we {will simply write}  $\Mod(\beins)$ for $\Mod_{\cT}(\beins)$.}
\hB
\end{rem}

We apply \cref{wtijgowergferwgwref} to the lax symmetric monoidal functor ${K_{*}}:\nCalg \to  \Ab^{\Z}$. The tensor unit of $\nCalg$ is the $C^{*}$-algebra $\C$. 
We get a commutative algebra $\KU_{*}:=K_{*}(\C)$ in $\CAlg(\Ab^{\Z})$ and, by specializing \eqref{sdvsdcsacd}, a refinement of the  functor  in  \eqref{fadfdsf} to a lax symmetric monoidal functor 
\begin{equation}\label{vsavdscasdcasdcd}
{K_{*}}\colon\nCalg\to \Mod(\KU_{*})\ .
\end{equation} 
On morphism sets it induces a map 
\[\Hom_{\nCalg}(A,B)\to \Hom_{\KU_{*}}(K_{*}(A),K_{*}(B))\ .\]
We note that $\KU_{*}\cong \Z[\beta,\beta^{-1}]$, where {$\beta$ is the Bott element and} has degree $2$. 
For every $C^{*}$-algebra $A$, the $2$-periodicity of $K_{*}(A)$ is implemented by the multiplication by $\beta$.

\begin{ex}
If $A$ is a   commutative $C^{*}$-algebra, then the multiplication map
$A\otimes A\to A$ is a homomorphism of $C^{*}$-algebras.  If $A$ is  {in addition} unital, then this together  with the unit $\C\to A$ turns $A$ into an {object of} $\CAlg(\nCalg)$. Since 
$K_{*}\colon \nCalg\to \Mod(\KU_{*})$ is lax symmetric monoidal, we get the commutative algebra $K_{*}(A)$  in $\Calg( \Mod(\KU_{*}))$.  Its multiplication map is given by the composition
\[K_{*}(A)\otimes_{\KU_{*}} K_{*}(A)\to K_{*}(A\otimes A)\to K_{*}(A)\ ,\]
where the first map  is the exterior multiplication of \eqref{afdsfadsfaewad}, and the second is induced by the multiplication of $A$. 
\hB 
\end{ex}

A natural question which turned out to be essential not only in Elliott's classification of AF-algebras \cite{Elliott_1976}, but also in the   recent papers in the classification program, is how much information about a morphism $f \colon A\to B$ in $\nCalg$ can be recovered from its induced map ${K}_{*}(f)\colon {K}_{*}(A)\to {K}_{*}(B)$, see e.g.\  \cite[Thm.~8.3.3 and Thm.~8.4.1]{rordam:classification} or \cite[Thm.~9.8]{CGSTW:hom}. The following example shows that the transition from $f$ to $K_*(f)$ is far from being lossless.

\begin{ex}\label{ergjoewrfgerwfwe} 
Let us consider the manifold $\R\P^2$ together with its canonical filtration by submanifolds $\R\P^0 \subseteq \R\P^1 \subseteq \R\P^2$. Here  $\R\P^0 \cong \ast$ is the base point  and $\R\P^1 \cong S^1$. Restriction to $\R\P^{1}$ induces   a surjection
\[C(\R\P^2,*)\stackrel{f}{\to} C(S^{1},*)\] 
of $C^{*}$-algebras, where $C(\R\P^2,*)$ and  $C(S^{1},*)$ denote the continuous functions on the respective spaces  which vanish at the base point.    The exact sequence 
\begin{equation}\label{sDaddasd}
0\to \ker(f)\to C(\R\P^2,*)\to  C(S^{1},*) \to 0
\end{equation}
 gives rise to the following long exact sequence of $K$-theory groups (starting with $K_{1}(\ker(f))$) 
 \begin{equation}\label{rfiojfoqwfdwdwqdqwedqewdq}
\cdots\to 0\to 0\xrightarrow{K_{1}(f)} \Z \xrightarrow{2} \Z\to \Z/2\Z\xrightarrow{K_{0}(f)} 0\to \cdots\ 
\end{equation}
In particular, the map $f$ induces the zero map on $K$-theory groups. Hence $K_{*}$ can not distinguish $f$ from the zero homomorphism $C(\R\P^{2},*)\to  C(S^{1},*)$.
\end{ex}

As highlighted by the example above, the homomorphism $K_*(f)$ induced by $f$ does not keep track of the boundary maps in the $6$-term exact sequence associated to $f$. In \cref{gwjiogfrefwef} we will see that this changes when we consider $K$-theory as a spectrum-valued functor instead.

\subsection{Spectrum valued operator $K$-theory}
  
Homotopy theory enters the discussion of $C^{*}$-algebra $K$-theory by interpreting the $\Z$-graded abelian group $K_{*}(A)$ as the collection of homotopy groups $\pi_{*}K(A)$ of a $K$-theory spectrum $K(A)$. In the following we give a brief summary of the necessary $\infty$-categorical input we shall use {to talk about this spectrum valued $K$-theory functor}.

We will not need to understand any specific model of the   category of spectra $\Sp$. 
The basic spectra encountered in the present paper arise  as  mapping spectra between objects of stable $\infty$-categories. So we must understand the category $\Sp$ as the natural home of these mapping spectra, {and we will do so momentarily. Before, we note that one can characterize $\Sp$  itself    as the initial stable cocomplete $\infty$-category \cite[\S 1.4.3, 1.4.4, 4.8.2]{HA}, a universal property which specifies it up to unique equivalence}.
{The $\infty$-category of spectra} is generated by an object $\bS$ called the sphere spectrum  {in the sense} that evaluation at $\bS$ induces an equivalence 
$$\Fun^{\colim}(\Sp,\bD)\to \bD$$ for any cocomplete stable $\infty$-category $\bD$,
where the superscript $\colim$ indicates the full subcategory of colimit preserving functors.  In addition, $\Sp$ is also complete, i.e.,  it admits all small limits.

The stability of a finitely {(co)complete  $\infty$-category like the $\infty$-category $\Sp$ can be characterized by the following properties:}
\begin{enumerate}
\item The initial object is also terminal (such categories are called pointed) and we denote an initial object by $0$.
\item The (therefore existing) canonical map from a finite coproduct to a finite product is an equivalence (such categories are called {semiadditive}), and we denote the binary coproduct functor by $\oplus$.
\item A square is a pushout if and only if it is a pullback.
\end{enumerate}

{To describe the relation between stable categories and spectra in general, we first remark the following.}
\begin{rem}\label{eiorghbergerg}
There is a more fundamental $\infty$-category $\Spc$ of spaces which is characterized similarly 
 as the  initial cocomplete $\infty$-category. It is generated by a final object $*$ {again in the sense that}
the evaluation at $*$ induces an equivalence 
$$\Fun^{\colim}(\Spc,\bD)\to \bD$$ for any cocomplete $\infty$-category $\bD$.
The $\infty$-categories of spaces and spectra are related by
an adjunction
$$\Sigma^{\infty}_{+}\colon \Spc\leftrightarrows \Sp\colon \Omega^{\infty}\ ,$$
The $\infty$-category of spaces is also complete and hence symmetric monoidal with respect to cartesian products. The $\infty$-category of $\Sp$ then admits a unique symmetric monoidal structure\footnote{classically this is the smash product of spectra and often denoted $\wedge$.} $\otimes$ with tensor unit $\bS$, for which $\Sigma^\infty_+$ refines to a symmetric monoidal functor. Consequently, $\Omega^\infty$ is canonically lax symmetric monoidal. {Furthermore, it has the following universal property: For any stable $\infty$-category $\cC$, the functor $\Omega^{\infty}$ induces an equivalence
\[ \Fun^{\mathrm{lex}}(\cC,\Sp) \stackrel{\simeq}{\to} \Fun^{\mathrm{lex}}(\cC,\Spc) \]
where the superscript lex stands from left exact, that is finite limit preserving, functors.}

By default, an $\infty$-category
$\cC$ is enriched in $\Spc$  in the sense that there is a mapping space bifunctor
$$\Map_{\cC}(-,-)\colon \cC^{\op}\times \cC\to \Spc\ , \quad (A,B)\mapsto \Map_{\cC}(A,B)\ .$$
 
For a stable $\infty$-category  $\cC$, {the universal property of $\Omega^{\infty}$ implies that}  the mapping space bifunctor canonically refines to a mapping spectrum bifunctor  {$\map_{\cC}$ such that there exists an essentially unique dashed arrow in the diagram}
\[\begin{tikzcd}[column sep=large]
	& \Sp \ar[d,"{\Omega^\infty}"] \\
	\cC^\op \times \cC \ar[r,"{\Map_\cC(-,-)}"'] \ar[ur, dashed, bend left, "{\map_\cC(-,-)}"] & \Spc
\end{tikzcd}\]
 {rendering it commutative.
As mentioned above, a natural way to construct spectra as objects in $\Sp$ is therefore to realize them as mapping spectra in stable $\infty$-categories.}\hB
\end{rem}

Any {pointed} $\infty$-category with finite colimits and a terminal object comes equipped with a suspension (or shift) endofunctor $\Sigma$ given by the push-out of the span of endofunctors $0 \leftarrow \id \to 0$.   Stability {of such an $\infty$-category} is equivalent to
 the {statement} that $\Sigma$ is an autoequivalence. Applying this to  $\Sp$, for 
 a spectrum $E$  {one}  defines  the homotopy groups
 $$\pi_{k}(E):=\Hom_{\ho(\Sp)}(\Sigma^{k}\bS,E)\ , \quad k\in \Z\ ,$$
 where $\ho(\Sp)$ denotes the homotopy category.

Forming homotopy groups of spectra gives rise to a lax
 symmetric monoidal functor
\[\pi_{*}\colon\Sp\to \Ab^{\Z}\ .\] This functor 
 is conservative (i.e.\ it detects equivalences) and preserves {products}, sums and filtered colimits, and satisfies
 $\pi_{*}\circ \Sigma\simeq \pi_{*-1}$. 
A fibre sequence in  stable $\infty$-category is a pushout (or equivalently a pullback)
  square as follows.\begin{equation}\label{fibre-sequence} \xymatrix{F\ar[d]\ar[r]&E\ar[d]\\0\ar[r]&B}
\end{equation} 
The fibre sequence  {determines, and} is determined by, 
the associated boundary map  
$ \partial \colon B\to \Sigma F$  {obtained via pasting pushout diagrams and using the definition of the suspension functor:}
\[\begin{tikzcd}
	F \ar[r] \ar[d] & E \ar[r] \ar[d] & 0 \ar[d] \\
	0 \ar[r] & B \ar[r] & \Sigma F
\end{tikzcd}\]

 If \eqref{fibre-sequence} is a fibre sequence of spectra, then
  the boundary map  induces a boundary map  $\partial \colon \pi_*(B) \to \pi_{*-1}(F)$ on the level of homotopy groups  {which makes} the sequence  
 \begin{equation}\label{sdcvscsacsdca}
 \dots \lto \pi_*(F) \to \pi_*(E) \to \pi_*(B) \xrightarrow{\partial} \pi_{*-1}(F) \to \dots\end{equation}
is exact.

\begin{rem}
A convenient way to manipulate objects in an $\infty$-category
$\cC$  is to take  {colimits and limits over suitable diagrams, provided $\cC$ admits these (co)limits. Particularly easy diagrams are constant diagrams, indexed over objects of $\Spc$ (thought of as $\infty$-groupoids). These give rise to}
tensor and power functors which again exist if $\cC$ is sufficiently cocomplete or complete.

The tensor is a bifunctor
 $$\cC\times \Spc\to \cC\ , \quad (C,X)\mapsto C\otimes X\ $$ which, in view of the universal property of $\Spc$ it
  is uniquely characterised {by the fact that it preserves colimits in the second variable and by setting $C\otimes *\simeq C$}.
Similarly, the power is a bifunctor
 $$\Spc^{\op}\times\cC\to \cC\ , \quad  (X,C)\mapsto C^{X} $$
again uniquely characterised by  {the fact that it sends limits (in $\Spc^{\op}$) in the  first variable  to limits and by setting }$C^{*} \simeq C$.  
For fixed $X$ in $\Spc$ we have an adjunction   \begin{equation}\label{dsacsdcdacsda}
-\otimes X:\cC\leftrightarrows \cC:(-)^{X}\ .
\end{equation}\hB
\end{rem}
 \begin{rem}\label{werijgwergwerfwef}
 For us, it is important to know that there is a functor 
\begin{equation}\label{sfdgsfdgfsgsfdg}
\ell:\Top\to \Spc
\end{equation}  representing  $\Spc$ as  the Dwyer-Kan localization of $\Top$ at the weak equivalences. We  sometimes use the expression of $\ell(X)$
in terms of  the category of singular simplices $\Simp(X)$  of $X$. Its objects are the continuous maps $\sigma\colon \Delta^{n}\to X$, and its morphisms
$\tau\to \tau'$ are maps $\phi:[n]\to [n']$ in $\Delta$ such that
$$\xymatrix{\Delta^{n}\ar[rr]^{|\phi|}\ar[dr]^{\tau}&&\Delta^{n'}\ar[dl]_{\tau'}\\&X&}$$ commutes.
We then have an equivalence   $$\ell(X)\simeq  \colim_{\Simp(X)}*$$ in $\Spc$. 

We can interpret the canonical functor $\Simp(X)\to \ell(X)$ as the
Dwyer-Kan localization  of  $\Simp(X)$ at all morphisms.

Let $\cC$ be a sufficiently complete and cocomplete $\infty$-category and
 $C$ be an object in $\cC$. For a   topological space $X$ we can form new objects
$C\otimes \ell(X)$ and $C^{\ell(X)}$ in $\cC$. Below  we will usually omit the $\ell$ from the notation in order to simplify.  Using the presentation of $\ell(X)$ above we get  equivalences
\begin{equation}\label{fdsvfvfdvsvffvr}
 C\otimes  X\simeq \colim_{\Simp(X)}\const_{C} \ , \quad 
 C^{X}\simeq  \lim_{\Simp(X)}\const_{C}\ ,
 \end{equation}  
where $\const_{C}$ denote the constant functor on $\Simp(X)$ with value $C$.
 
If $\cC$ is stable, then the tensor and cotensor extend along $\Sigma^{\infty}_{+}$ to bifunctors
$$\cC\times \Sp\to \cC\ , \quad (C,X)\mapsto C\otimes X$$ 
 and
$$\Sp^{\op}\times\cC\to \cC\ , \quad  (X,C)\mapsto C^{X}\ . $$
with similar compatibility with colimits and limits, respectively.  

If $E$ is in $\Sp$, then the    homotopy groups $\pi_{*}(E\otimes X)$ and $\pi_{*}(E^{X})$  can be calculated by Atiyah-Hirzebruch spectral sequences with second pages
\begin{equation}\label{casdcsdaccsca1}
E_{2}^{p,q}\cong H_{p}(X,\pi_{q}E) \quad  \Rightarrow \quad \pi_{p+q}(E\otimes X) \end{equation} for the tensor and  \begin{equation}\label{casdcsdaccsca}
E_{\ulrich{2}}^{p,q}\cong H^{p}(X,\pi_{-q}E)\quad  \Rightarrow \quad \pi_{-p-q}(E^{X})
\end{equation}
  for  the power. 
\hB\end{rem}

We now come back to $C^{*}$-algebra $K$-theory with the following proposition:

\begin{prop}\label{wekojgwerferfwe}
There exists an essentially unique   functor
$$K\colon \nCalg\to \Sp$$
which is  homotopy invariant, $\mathbb{K}$-stable,  preserves filtered colimits and  
sends short exact sequences of $C^{*}$-algebras to fibre sequences
  such that 
\begin{equation}\label{ewfqwefqewfeqdqwedewdq}\xymatrix{\nCalg\ar[rr]^{K}\ar[dr]^{K_{*}}&&\Sp\ar[dl]_{\pi_{*}}\\&\Ab^{\Z}&}
\end{equation} 
commutes.\footnote{ {{By \cref{werjigoegsfg}} 
 it suffices to fix an isomorphism between $K_0$ and $\pi_0K$.}}
In addition, this functor has an essentially unique lax symmetric monoidal refinement.
\end{prop}
\begin{proof} 

The classical proof of the existence of a functor as asserted in  \cref{wekojgwerferfwe} goes via an explicit construction using some point-set model for spectra {together with a} verification of the 
asserted properties, see e.g.\ \cite{joachimcat}. 
In contrast,   in  \cref{wekojgwerferfwe1} we {give an abstract definition of a functor using a stable categorical enhancement of Kasparov's triangulated KK-category, see \cref{qrijgfqowefdwefqwefqwedqwed} for a slightly refined version. By construction, this functor is homotopy invariant,  $\mathbb{K}$-stable, sends exact sequences with  completely positive splits to fibre sequences (i.e.\ it is semi-exact), $s$-finitary (see \cref{wjigriowrtgwefrefwerf}),  and it fits into the triangle \eqref{ewfqwefqewfeqdqwedewdq}. We will then show in \cref{werjigoegsfg} that it is already uniquely determined by these properties. The existence and uniqueness of the lax symmetric monoidal refinement
will be shown in  \cref{gqerhgperfqrferfwrf}.}
  
In the following we argue, using  {well-known properties} of the classical group-valued $K$-theory functor,
that the spectrum valued $K$-theory functor defined in this way is not only semi-exact, but exact, and that it not only $s$-finitary but  preserves all filtered colimits.
 
  In order to see that $K$ preserves filtered colimits we simply note that $\pi_{*}:\Sp\to \Ab^{\Z}$ is conservative  and preserves sums and filtered colimits, and  that $K_{*}$  preserves sums and filtered colimits. It then remains to show that $K$ sends all exact sequences  \eqref{advasdcasdca} 
to fibre sequences.
We form the following map of short exact sequences \begin{equation}\label{adscadscasdccd}
\xymatrix{0\ar[r]&A\ar[r]\ar[d]^{\iota}&B\ar[r]^{f}\ar[d]^{h}&C\ar@{=}[d]\ar[r]&0\\0\ar[r]&
C(f)\ar[r]&Z(f)\ar[r]^{\tilde f}&C\ar[r]&0}\ .
\end{equation}
  where $Z(f)$ and $C(f)$ are the mapping cylinder and mapping cone of $f$, respectively. The lower exact sequence is {semi-split} and
the map $h$ is a homotopy equivalence.  We now apply $K$ and get the commutative diagram
 $$\xymatrix{ &K(A)\ar[r]\ar[d]^{K(\iota)}&K(B)\ar[r]^{K(f)}\ar[d]^{K(h)}_{\simeq}&K(C)\ar@{=}[d] & \\ &
K(C(f))\ar[r]&K(Z(f))\ar[r]^{K(\tilde f)}&K(C)&}$$ in $\Sp$.
 Since $K$ is semi-exact (see \cref{wetkjgowergwerfw9}), the lower sequence is  a fibre sequence. 
 It therefore suffices to show that $K(\iota)$ is an equivalence. 
 To this end we apply $K_{*}$ to the diagram \eqref{adscadscasdccd} and consider the resulting map of long exact sequences. 
 Since $K_{*}(h)$ is an isomorphism by the homotopy invariance of $K_{*}$, the $5$-lemma implies that
 $K_{*}(\iota)$ is an isomorphism, too.  {This implies that $K(\iota)$ is an equivalence 
since $\pi_{*}$ is conservative  and there is an isomorphism $K_* \cong \pi_*K$.}
 \end{proof}

\begin{rem} 
 {The fact that the $\K$-theory functor sends all short exact sequences of $C^*$-algebras~\eqref{advasdcasdca} to fibre sequences of spectra  and thus encodes the boundary maps of the exact $6$-terms sequences in a natural way is one of the advantages 
 of the spectral picture over the group-valued $K$-functor.}
 {W}riting the sequence  \eqref{advasdcasdca} as a square
$$\xymatrix{A\ar[r]\ar[d] &B \ar[d] \\ 0\ar[r] & C} $$ and applying $
K$ we get a commutative  square
$$\xymatrix{K(A)\ar[r]\ar[d] &K(B) \ar[d] \\ 0\ar[r] & K(C)} $$ 
whose filler encodes the hidden data of the fibre sequence
 \begin{equation}\label{adfadfadfadf}K(A)\to K(B)\to K(C)\ .
\end{equation} 
 The boundary operator of the long exact sequence  
of homotopy groups \eqref{sdcvscsacsdca}  can be derived from this fibre sequence.  
This is  in contrast to the boundary operator in \eqref{cdscasdcasdcasdcdc} which has to be constructed in addition to the functor $K_{*}$.

Note that we have not {claimed or }shown that the boundary operator in \eqref{cdscasdcasdcasdcdc} coincides with the  one derived from the fibre sequence in \eqref{adfadfadfadf}. 
\hB
\end{rem}
 
Since $K$ is lax symmetric monoidal it  follows that $\KU:=K(\C)$ refines  to a commutative   algebra object in
$\Sp$, i.e.\ an object $\CAlg(\Sp)$. It is equivalent to the commutative algebra that topologists call $\KU$, constructed   as the  periodic version of the group completion of the topological category  of $\mathbb{C}$-vector spaces.
 For every
$C^{*}$-algebra $A$  the $K$-theory spectrum $K(A)$   naturally has the structure of a module over $\KU$.
The $\infty$-categorical analogue of~\eqref{sdvsdcsacd} gives the following refinement of  \cref{wekojgwerferfwe}
  \begin{cor}\label{rwefefefw}
  The spectrum-valued $K$-theory functor for $C^{*}$-algebras naturally refines to a lax symmetric monoidal functor $$K\colon \nCalg\to \Mod(\KU)$$
such that   
$$\xymatrix{\nCalg\ar[rr]^{K}\ar[dr]^{K_{*}}&&\Mod(\KU)\ar[dl]_{\pi_{*}}\\&\Mod(\KU_{*})&}$$
commutes. 
\end{cor}

\section{Applications of spectrum-valued $K$-theory}

In this subsection we illustrate some aspects of the spectrum valued $K$-theory of $C^{*}$-algebras by various examples.

 \subsection{Detection of maps} \label{gwjiogfrefwef}
 We  revisit   \cref{ergjoewrfgerwfwe}.
The exact sequence   \eqref{sDaddasd} induces a fibre sequence
$$K(\ker(f))\to K(C(M,*))\xrightarrow{K(f)}  K(C(S^{1},*))$$
in $\Sp$ which is  determined up to equivalence by 
the map $K(f)$. The non-triviality of the boundary operator in \eqref{rfiojfoqwfdwdwqdqwedqewdq} (the multiplication by $2$ in the middle)
shows that $K(f)$ is not equivalent to the zero map, in contrast to the vanishing of the map $K_{*}(f)$.

\subsection{Künneth formulas} 
 \label{rgjhiergewrg9}
Given $C^{*}$-algebras $A$ and $D$
   the  {K\"unneth problem is the quest for a K\"unneth formula which calculates $K_{*}(A\otimes D)$ in terms of $K_{*}(A)$ and $K_{*}(D)$.}
  
  The lax symmetric structure on the  $\Mod(\KU_{*})$-valued version of $K_{*}$ gives for any $A$ and $D$  in $\nCalg$ a  factorization of the    {exterior multiplication map}  \eqref{afdsfadsfaewad} over a  homomorphism \begin{equation}\label{kuenneth-map}
K_{*}(A)\otimes_{\KU_{*}} K_{*}(D)\to K_{*}(A\otimes D)\ ,
\end{equation}
 and one could ask how far this homomorphism  is from being an isomorphism. 
  
 If $A\cong \C$, then it is {tautologically} an isomorphism for any $D$. 
 Since the functors $K_{*}$ and $-\otimes_{\KU_{*}}K_{*}(D)$ preserve  sums and filtered colimits, the class of algebras $A$ for which    \eqref{kuenneth-map} is an isomorphism for any $D$ is closed under the formation of sums and filtered colimits. {A further closure property would be 2-out-of-3 for exact sequences: Given a short exact sequence of $C^*$-algebras where two of the algebras are in the class for which  \eqref{afdsfadsfaewad} is an isomorphism for all $D$, is the same true for the third algebra? Fixing such a test $C^*$-algebra $D$, and using the long exact sequence of $K$-groups \eqref{cdscasdcasdcasdcdc}, this follows from the 5-lemma if $K_*(D)$ is a flat $\KU_*$-module, since in this case, the functor $-\otimes_{\KU_*} K_*(D)$ preserves exact sequences.}
 The classical way to generalize to  the case where $K_{*}(D)$ is not necessaily flat is to give up to   require that 
 \eqref{afdsfadsfaewad} is an isomorphism, but to require instead a short exact   sequence of the form \begin{equation}\label{adfsfasdfds}
0 \to K_{*}(A)\otimes_{\KU_{*}} K_{*}(D)\to K_{*}(A\otimes D)\to   \Tor^{\KU_{*}}(K_{*-1}(A),K_{*}(D))\to 0\ . 
\end{equation}
This is equivalent to replacing the tensor product over $\KU_*$ on the left hand side of \eqref{kuenneth-map} by the derived tensor product over $\KU_*$. Note that  the  graded ring $\KU_*$ has homological dimension 1, so indeed no higher Tor groups ought to appear.

 {Working with 
  $\KU$-modules  instead of $\KU_{*}$-modules gives the following}.
There is a short exact sequence calculating the homotopy groups of a tensor product of $\KU$-modules as follows
  \begin{equation}\label{adfsfasdfds1}
0 \to K_{*}(A)\otimes_{\KU_{*}} K_{*}(D)\to \pi_{*}(K(A)\otimes_{\KU} K(D))\to   \Tor^{\KU_{*}}(K_{*-1}(A),K_{*}(D))\to 0\ , 
\end{equation}
also known as the K\"unneth sequence.\footnote{This is a special case of a K\"unneth spectral sequence converging to $\pi_*(M\otimes_R N)$ for any ring spectrum $R$ and $R$-modules $M$ and $N$.}
So the K\"unneth problem can be phrased as whether for a given $A$ the map \begin{equation}\label{dafadfasdfas}
K(A)\otimes_{\KU} K(D)\to K(A\otimes D)
\end{equation}  induced by the lax symmetric monoidal structure of $K$
   is an equivalence for all $C^{*}$-algebras~$D$. If this is the case, we say that $A$ satisfies the K\"unneth formula.

 Since $ {K(-)}\otimes_{\KU}K(D)$ preserves sums, filtered colimits and fibre sequences,
 it is obvious that the class of algebras which   satisfy the K\"unneth formula
 contains $\C$ and is closed under sums, filtered colimits,  {retracts}, and satisfies the two-out-of three property for exact sequences. See also  \cref{qirjoqrfewqedew}.

 {In the above presentation of the K\"unneth problem, we have replaced the condition of existence and exactness of the sequence \eqref{adfsfasdfds} by the condition that the (a priori existing) map in \eqref{dafadfasdfas} is an equivalence}. In the present situation this might  {not} seem  {like a big} improvement, but again note that the exact sequence \eqref{adfsfasdfds1} {is the ``correct term''} only since $\KU_{*}$ has homological dimension one.  There are variants of the theory, e.g.\ real $K$-theory, equivariant $K$-theory, $K$-theory for $X$-algebras where $X$ is a topological space (or combinations of these situations)
where the analogue of $\KU_{*}$ does not have homological dimension one.
But the interpretation of the K\"unneth formula as the property that the map in 
\eqref{dafadfasdfas} is an equivalence makes sense in all these variants.

 \subsection{Swan's theorem}\label{wegojwoergferfweferfwr}
 In this example we give a formulation  \ulrich{of} Swan's theorem in the homotopy theoretic context of  spectrum-valued $K$-theory. For a $C^{*}$-algebra $A$ and a compact Hausdorff space $X$ we  consider the  $C^{*}$-algebra $C(X,A)$ of continuous $A$-valued functions on $X$. We then form the $\KU$-module $K(C( {X},A))$ in $\Mod(\KU)$.  We will see that  if  $X$ is homotopy finitely dominated, i.e.\ a retract, up to homotopy, of a finite CW complex, then all  information for the calculation of the $\KU$-module $K(C(X,A))$ is contained in the $\KU$-module  $K(A)$ and the space $\ell(X)$ in $\Spc$.

We first recall the classical formulation for $A=\C$. If $X$ is a  compact Hausdorff space, then $C(X)$ denotes the commutative $C^{*}$-algebra of continuous
functions on $X$.   Swan's theorem then asserts that
the Grothendieck group of projective modules over $C(X)$ is isomorphic to
the topological $K$-theory group  of $X$ defined in terms of vector bundles, or equivalently,   as the group of homotopy classes
$$K^{0}(X):= [X, \Z\times BU]$$ of maps from $X$ to the $h$-space $\Z\times BU$. 
In formulas, we have an isomorphism
\begin{equation}\label{adsfadsfadsfsafds}
K_{0}(C(X))  \cong  K^{0}(X) \ .
\end{equation} 
Note that this is actually an isomorphism  of rings where the product on the right-hand side  is induced by the tensor product of vector bundles.

We now use the notation from   \cref{eiorghbergerg}. Let $X$ be  any topological space.
 If $E$   is a
 commutative algebra in $\Mod(\KU)$, then so is $E^{X}$. 
 The product on $E^{X}$ involves the 
 multiplicative structure of $E$ and the diagonal of $X$.
{Since the spectrum valued $K$-theory functor is lax symmetric monoidal}, for any     unital and commutative $C^{*}$-algebra $A$ we get   commutative algebras $K(A)$ and $K(C(X,A))$
 in $\Mod(\KU)$.

In the proof of the \cref{weokjgopwegfrefwerf} below, we will use the following homotopy theoretic consideration about the relation between the $1$-category of compact Hausdorff spaces and the $\infty$-category $\Spc$.
Let   $\Top_{\mathrm{hCW}}$ denote the full sub-category of  $\Top$ consisting of spaces  of the homotopy type of CW-complexes. Its Dwyer-Kan localization  \begin{equation}\label{weqfqwefqfqewdewqd}\ell\colon \Top_{\mathrm{hCW}}\to \Spc
\end{equation} 
 at the homotopy equivalences\footnote{Note the difference to \eqref{sfdgsfdgfsgsfdg} where we localize the bigger category of all topological spaces at the larger class of weak equivalence.} is one of the  standard presentations
 of the {$\infty$}-category of spaces. 
 An object of $\Top_{\mathrm{hCW}}$ is  homotopy finitely dominated if it is a homotopy retract of a finite CW-complex.
 The restriction  $$\ell \colon\Top^{\mathrm{fd}}_{\mathrm{hCW}}\to \Spc^{\omega}$$
 of  \eqref{weqfqwefqfqewdewqd} to the full subcategory  of homotopy finitely dominated spaces presents the full subcategory 
  compact objects in $\Spc$ as a localization, too. Note that $\Spc^{\omega}$ is generated from the final object $*$ by finite colimits and retracts. Thus, if $\bM$ a finitely cocomplete and idempotent complete $\infty$-category, then evaluation at $*$ induces an equivalence
 \begin{equation}\label{qwfefewqdwedqwedqewdqewd} \Fun^{\Rex}(\Spc^{\omega},\bM)\stackrel{\simeq}{\to} \bM\ ,  \end{equation}  
  where the superscript  $\Rex$ stands for right exact, i.e.\ finite colimit preserving functors.

 Let $\CH^{\mathrm{fd}}_{\mathrm{hCW}}$ denote the full subcategory of compact Hausdorff spaces in  
  $\Top^{\mathrm{fd}}_{\mathrm{hCW}}$.  Let $X$ be in $\CH^{\mathrm{fd}}_{\mathrm{hCW}}$ 
  and $(Y,Z)$ be a pair of closed subspaces of $X$ such that $Y\to Z$ is a cofibration and $Y\cup Z=X$.
  Then \begin{equation}\label{qewfhq8wefhqwefqwedqewdqewd22} \xymatrix{ Y\cap Z \ar[r]\ar[d]&  Z \ar[d]\\ Y \ar[r]& X } 
\end{equation}
is a push-out square in $\CH^{\mathrm{fd}}_{\mathrm{hCW}}$ and $\ell$ sends this square to a push-out square in $\Spc^{\omega}$.
 We will  call $(Y,Z)$ a cofibrant decomposition of $X$.

  We consider two functors
    $\tilde F,\tilde F'\colon\CH^{\mathrm{fd}}_{\mathrm{hCW}}\to \bM$ to an idempotent complete finitely cocomplete target.
  \begin{lem}\label{wrgwtrgwrgfwerfwerf}
  If $\tilde F$ and $\tilde F'$ are homotopy invariant,  send the empty space to an initial object and the squares 
  \eqref{qewfhq8wefhqwefqwedqewdqewd22} for cofibrant decompositions to push-out squares,
  then an equivalence $\tilde F(*)\simeq \tilde F'(*)$ essentially uniquely extends to an equivalence of functors
  $\tilde F\simeq \tilde F'$.
  \end{lem}
\begin{proof}
Let $\Sp^{\fin}$ be the full subcategory of finite spaces generated by $*$ under finite colimits. The inclusion $\Spc^{\fin}\to \Spc^{\omega}$  presents its target as an idempotent completion. 
We let $\CH^{\mathrm{fin}}_{\mathrm{hCW}}$ denote the full subcategory of
$\CH^{\mathrm{fd}}_{\mathrm{hCW}}$ of compact Hausdorff spaces which are homotopy equivalent to
finite CW-complexes. The restriction of \eqref{weqfqwefqfqewdewqd} gives a functor
$\ell:\CH^{\mathrm{fin}}_{\mathrm{hCW}}\to \Spc^{\fin}$ which again presents its target as the Dwyer-Kan localization at the homotopy equivalences.
We let furthermore $\Spc^{\omega,CH}$  denote the full subcategory of $\Spc^{\omega}$ given by the essential image $\CH^{\mathrm{fd}}_{\mathrm{hCW}}$ under
the functor $\ell$ from \eqref{weqfqwefqfqewdewqd}. The functor
$\ell:\CH^{\mathrm{fd}}_{\mathrm{hCW}}\to \Spc^{\omega,CH}$ is still a   localization
at the homotopy equivalences. Since every finite CW-complex is a compact Hausdorff space we have inclusions
$$\Spc^{\fin}\subseteq \Spc^{\omega,CH}\subseteq \Spc^{\omega}$$ which induce  equivalences of idempotent completions. In particular, for any idempotent complete (and finitely cocomplete)    $\infty$-category $\bM$ the restrictions along these inclusions induce equivalences  
\begin{equation}\label{verewvewrwfrfefw} \Fun^{(\Rex)}( \Spc^{\omega},\bM)\stackrel{\simeq}{\to} \Fun^{(\Rex)}( \Spc^{\omega,CH},\bM) \stackrel{\simeq}{\to}\Fun^{(\Rex)}( \Spc^{\fin},\bM) \end{equation} 
of categories of functors (or finite colimit preserving functors, respectively).

Since $\tilde F$ and $\tilde F'$ are homotopy invariant they
descend through $\ell$ to functors $$F,F': \Spc^{\omega,CH}\to \bM\ .$$
Since every push-out square in $\Spc^{\fin}$ is the image under $\ell$ of  a cofibrant decomposition of a finite
CW-complex into subcomplexes we can conclude that the restrictions of $F$ and $F'$ to $\Spc^{\fin}$ preserve finite colimits.  In view of \eqref{verewvewrwfrfefw} we can further conclude that $F,F'$   have essentially unique
extensions   to finite colimit preserving functors   
$ \hat F,\hat F':\Spc^{\omega}\to \bM$.  In view of \eqref{qwfefewqdwedqwedqewdqewd} the equivalence
$$  F(*)\simeq \tilde F(*)  \simeq \tilde   F'(*)\simeq   F'(*)$$ essentially uniquely extends to
an equivalence $\hat F\simeq \hat F'$. 
Using  \eqref{verewvewrwfrfefw} again we get an equivalence $\tilde F\simeq \tilde F'$ and therefore the desired  equivalence
$F\simeq F'$ extending the given equivalence of values on the point.
\end{proof}

\begin{prop}\label{weokjgopwegfrefwerf}
If the compact Hausdorff space $X$ is  {homotopy finitely dominated}, then we have a  natural equivalence    
\begin{equation}\label{adsvdscdscacdscascadsc}
K(C(X,A))\simeq K(A)^{X}\ .
\end{equation}
in $\Mod(\KU)$. If $A$ is unital and commutative, then  \eqref{adsvdscdscacdscascadsc} refines to an equivalence of commutative algebras in $\Mod(\KU)$.
\end{prop}
\begin{proof}
We wish to use \cref{wrgwtrgwrgfwerfwerf} for the two functors $ X\mapsto K(C(X,A))$ and $X\mapsto K(A)^{X}$ from $\CH^{\mathrm{fd}}_{\mathrm{hCW}}$ to $\Mod(\KU)^{\op}$. 
   These functors   are both homotopy invariant and send the empty space to zero.  Furthermore,    $\Mod(\KU)^{\op}$ is cocomplete and hence also idempotent complete.

   For a  $\KU$-module $C$  the   power functor $\Spc\ni  {B}\to C^{ {B}}\in \Mod_{\Sp}(\KU)^{\op}$    preserves  push-outs. 
  Since the image of \eqref{qewfhq8wefhqwefqwedqewdqewd22} under $\ell$
 is a push-out for every cofibrant   decomposition $(Y,Z)$ of $X$ in $\CH^{\mathrm{fd}}_{\mathrm{hCW}}$ the functor 
 $ \K(A)^{-}$  sends cofibrant decompositions to push-outs.      We now argue, that $ \K(C(-,A))$ has this property, too.
 
 We let $(Y,Z)$ be a closed decomposition of  $X$ in $\CH^{\mathrm{fd}}_{\mathrm{hCW}}$. 
   The commutative square \eqref{qewfhq8wefhqwefqwedqewdqewd22}
 then  induces a   commutative  square 
      \begin{equation}\label{adfasdfadfadsf}
\xymatrix{K(C(X,A))\ar[r]\ar[d]&K(C(Z,A))\ar[d]\\ K(C(Y,A))\ar[r]&K(C(Y\cap Z,A))}
\end{equation}   in $\Mod(\KU)$. We must show that
  the latter   is cartesian.        It suffices to show the induced map of fibres of the horizontal maps is an equivalence. To this end 
 we use that $\K$ sends the exact sequences \begin{equation}\label{fdvdvfsvdfv}
0\to \ker(C(X,A)\to C(Y,A))\to C(X,A)\to C(Y,A)\to 0
\end{equation}
   and \begin{equation}\label{fdvdvfsvdfv1}
0\to  \ker(C(Y,A)\to C(Y\cap Z,A))\to C(Y,A) \to C(Y\cap Z,A)\to 0
\end{equation}
  to fibre sequences, 
 and the isomorphism of $C^{*}$-algebras $$\ker(C(X,A)\to C(Y,A))\cong C_{0}(X\setminus Y,A)\cong C_{0}(Z\setminus (Y\cap Z),A)\cong \ker(C(Z,A)\to C(Y\cap Z,A))\ . $$
  If $A$ is unital and commutative, the same argument applies when {changing the codomain to the category $\CAlg(\Mod(\KU))^{\op}$.}
\end{proof}

In the case of $A=\C$,  applying $\pi_{*}$ and using that $\Omega^{\infty}\KU\simeq \Z\times \mathrm{BU}$ as $h$-spaces we can deduce \eqref{adsfadsfadsfsafds} from \eqref{adsvdscdscacdscascadsc} provided $X$ is  in $\CH^{\mathrm{fd}}_{\mathrm{hCW}}$. 
     
\begin{ex} 
In the formulation  of   \cref{weokjgopwegfrefwerf},  the restriction to homotopy finitely dominated spaces   is necessary. 
Consider for example the compact subspace
$$X:=\{1/n\mid n\in \nat\}\cup \{0\}$$ of the unit interval in $\R$. 
 Then $\ell(X)\simeq \coprod_{\nat} \ell(*)$ so that $K(A)^{X}\simeq \prod_{\nat} K(A)$.
One the other hand one can  {show} that $K_{*}(C(X,A))$ is the subgroup of
$\prod_{\nat} K_{*}(A)$ of eventually constant sequences.
Thus in this case $K(C(X,A))\not\simeq K(A)^{X}$. \hB
\end{ex}

We consider the right-hand side of \eqref{adsvdscdscacdscascadsc} as a homotopy theoretical formula for $K(C(X,A))$. This formula can be employed to  calculate the  $K$-theory groups   $K_{*}(C(X,A))$   using a Atiyah-Hirzebruch spectral sequence  \eqref{casdcsdaccsca} with second term
$$E_{2}^{p,q}:=H^{p}(X, K_{-q}(A))\ .$$

Using the $\KU$-module structure on $K(A)$ we can rewrite the right-hand side of  \eqref{adsvdscdscacdscascadsc} in the form
$$K(A)^{X}\simeq   \map_{\KU}(\KU\otimes X,K(A))\ $$
 {where $\map_\KU(-,-)$ is short for the $\KU$-module $\map_{\Mod(\KU)}(-,-)$}.
Sometimes the $\KU$-module $ \KU\otimes X$ has a very simple structure,  {as in the following example.}
\begin{ex}It is known\footnote{For instance as a consequence of the complex orientability of $\KU$.} that
$$\KU\otimes \C\P^{n}\simeq \bigoplus_{i=0}^{n} \Sigma^{2n}\KU\ .$$
This immediately implies
$$\K(C(\C\P^{n},A))\simeq \bigoplus_{i=0}^{n}\Sigma^{-2n}\K(A)$$
in $\Mod(\KU)$. \hB
 \end{ex}
If $X$ is in $\CH^{\mathrm{fd}}_{\mathrm{hCW}}$, then one can also rewrite the right-hand side of  \eqref{adsvdscdscacdscascadsc}
as\footnote{Both sides are compatible with colimits in $X$, and they agree for $X=\ast$.}  $$K(A)^{X}\simeq \KU^{X}\otimes_{\KU} K(A)\ .$$ From this formula we see that all information in order to calculate $\K(C(X,A))$ is contained in the $\KU$-module
$\K(C(X))\simeq \KU^{X}$, and that we have the exact sequence
$$0\to K^{-*}(X)\otimes_{\KU_{*}} K_{*}(A) \to K_{*}(C(X,A))\to \Tor^{\KU_{*}}_1(K^{-*{+1}}(X),K_{*}(A))\to 0$$
known from   \cref{rgjhiergewrg9} since $C(X)$ satisfies the K\"unneth formula.  
Generalizations of this example will be discussed in    \cref{adsfqdewdfew} and    \cref{wkogpwgerfwerfwf}.

\section{Higher structures on $KK$-theory}
\subsection{Classical  $KK$-theory}\label{wekojgwerferfwe111}

The classical $KK$-groups $KK_{0}(A,B)$ of two separable $C^{*}$-algebras are defined as groups of equivalence classes of $(A,B)$-Kasparov modules 
\cite{kasparovinvent} (see also the textbook  \cite{blackadar}) or quasi-homomorphisms $A\rightsquigarrow B$ \cite{MR899916}. They are bivariantly functorial in $A$ and $B$, that is contravariant in $A$ and covariant in $B$. Furthermore, one has an associative bilinear composition product, the Kasparov product
$$KK_{0}(B,C)\otimes KK_{0}(A,B)\to KK_{0}(A,C)\ .$$
Following \cite{Higson87} (see also  \cite[Sec.22.1]{blackadar}) it is natural to capture these constructions in terms of a functor 
\begin{equation}\label{vsvscsacsdca}
\kk_{\sepa}^0\colon \nCalg_{\sepa}\to \KK_{\sepa}^0
\end{equation} 
 from separable $C^{*}$-algebras to an additive category $\KK_{\sepa}^0$ such that
 $$KK_{0}(A,B)= \Hom_{\KK^{0}_{\sepa}}(\kk^{0}_{\sepa}(A),\kk^{0}_{\sepa}(B))\ .$$ 
Specializing at $A=\C$, there is in addition a canonical isomorphism of functors
\begin{equation}\label{fdsvsoihwjiovwefvwf}KK_{0}(\C,-) \cong K_0(-):\nCalg_\sepa \to \Ab\ ,
\end{equation} 
see \cite[Prop.\ 17.5.5]{blackadar}.

A  functor $F$ from  $C^{*}$-algebras  to an additive category can have the following properties:
\begin{enumerate}
	\item  homotopy invariance: If $f_{0}$ and $f_{1}$ are homotopic maps between $C^{*}$-algebras, then $F(f_{0})$ and $F(f_{1})$ are  equal.
	\item $\mathbb{K}$-stability:  The upper left corner conclusion
$A\to \mathbb{K}\otimes A$ induces an isomorphism $F(A)\to F( \mathbb{K}\otimes A)$    for any $C^{*}$-algebra $A$.
\item  split-exactness (for functors with an additive target): For every
exact sequence \begin{equation}\label{ergwergerferffw}
 0\to A\to B\xrightarrow{f} C\to 0
\end{equation}   of $C^{*}$-algebras
such that $f$ admits a right-inverse (we say that the sequence is split-exact), the sequence  $$0\to F(A)\to F(B)\stackrel{F(f)}{\to} F(C)\to 0$$
is split-exact.
\end{enumerate}
These properties are also well-defined for functors just defined on $ \nCalg_{\sepa}$.
The functor $\kk_\sepa^0$ of \eqref{vsvscsacsdca} can essentially uniquely be characterized by the following universal property:
\begin{theorem}[{\cite{MR1068250}}]\label{weriojgweorgewrgfwerf9}
The functor $\kk_{\sepa}^0\colon \nCalg_{\sepa} \to \KK_{\sepa}^0$ is   
initial among functors from
$\nCalg_{\sepa}$ to additive categories which are homotopy invariant, $\mathbb{K}$-stable and
split-exact. 
\end{theorem}

\begin{rem}
In addition, Higson showed that any $\mathbb{K}$-stable and split-exact functor to an additive category is already homotopy invariant \cite{higsa}.
\end{rem}

Any exact sequence \eqref{ergwergerferffw} gives rise to a diagram
\begin{equation}\label{adscadscasdffccd}
\xymatrix{ &A\ar[r]\ar[d]^{\iota}&B\ar[r]^{f}\ar[d]^{h}&C\ar@{=}[d] & \\ S(C)\ar[r]&
C(f)\ar[r]&Z(f)\ar[r]^{\tilde f}&C&}\ .
\end{equation}
where $S(C):=C_{0}(\R,C)$  and the lower sequence is a segment of the Puppe sequence associated to $f$.

\begin{theorem}[{\cite{MR2193334}}]   \label{goeirererfwe} The category $\KK_{\sepa}^0$  has a triangulated structure with shift operator $S(-)$\footnote{Note that $S(-)$ rather corresponds to the inverse of the shift in the usual conventions.} and whose  {distinguished} triangles are generated by the images of the Puppe sequences for maps $f$  under $\kk_{\sepa}^0$.
\end{theorem}

Split-exactness of $\kk_{\sepa}^0$  can now  be reformulated as the assertion that the map 
\[\kk_{\sepa}^0(\iota) \colon \kk_{\sepa}^0(A) \to \kk_{\sepa}^0(C(f)) \]
is an isomorphism for every split-exact exact sequence 
 \eqref{ergwergerferffw}, where $\iota$ is as in \eqref{adscadscasdffccd}. 
 
 \begin{definition}
 The sequence \eqref{ergwergerferffw} is called 
 semi-{split} if  $f$ admits a completely positive contractive (cpc) right-inverse.
 \end{definition}
 We note that a cpc map is not required to be a $*$-homomorphism and conversely, $*$-homomorphisms are cpc. Therefore, every split-exact sequence is semi-split, but the converse is not true. For instance, the surjections $\tilde{f} \colon Z(f) \to C$ as in \eqref{adscadscasdffccd} always admit cpc right inverses.
 The following is a deep result in $KK$-theory.
 
 \begin{theorem}[{\cite{zbMATH03727013}, \cite{zbMATH03973625}}]\label{rthojirthrtegetrg}
 The functor $\kk_{\sepa}^0$ is semi-exact in the sense that $\kk_{\sepa}^0(\iota)$ is an isomorphism for every semi-split exact sequence \eqref{ergwergerferffw}.
 \end{theorem}  In other words, $\kk_\sepa^0$ sends semi-split exact sequences to distinguished triangles.  

\begin{rem}\label{egojwpegerferfw}
It is known that the functor $\kk_\sepa^0$ is not exact, i.e., it does not send all exact sequences of $C^*$-algebras to distinguished triangles. 
Indeed, as shown in  Skandalis \cite{zbMATH00027188}, for lattices $\Gamma$ in $\mathrm{Sp(n,1)}$, the short exact sequence associated to the canonical surjection $C^*(\Gamma) \to C^*_r(\Gamma)$ is not sent to a distinguished triangle by $\kk_\sepa^0$.  
 \end{rem}

\begin{theorem}[{\cite{MR2193334}}]  \label{weklporgwergfwerfw} The  category $\KK_{\sepa}^0$  is canonically tensor triangulated   such that
$\kk_{\sepa}^0$ refines to a symmetric monoidal functor.
 \end{theorem}
 
 \begin{rem}
Note that in contrast to a lax symmetric monoidal functor, the assertion that $\kk_{\sepa,0}$ is symmetric monoidal expresses the facts that this functor preserves the tensor unit and that the monoidal  structure maps
$$\kk_{\sepa}^0(A)\otimes \kk_{\sepa}^0(B)\to \kk_{\sepa}^0(A\otimes B)$$
 are isomorphisms for all $A,B$ in $\nCalg_{\sepa}$. 
Recall that for this survey we consider the maximal tensor product on $\nCalg$, but 
 \cref{weklporgwergfwerfw} also has a version for the minimal tensor product\footnote{{In fact, \cite{MR2193334} considered only the minimal tensor product.}}. 
 \hB
 \end{rem}

  Using the shift operator  $S$ of the triangulated structure on $\KK_{\sepa}^{0}$ for any two objects $A,B$ in $\KK_{\sepa,0}$ we  obtain $\Z$-graded $KK$-groups by
 $$  KK_{n}(A,B):=\Hom_{\KK_{\sepa}^0}(  A ,S^{n}(B))$$  for all $n$ in $\Z$.
  The composition in $\KK_{\sepa}^0$ canonically
extends to a morphism of $\Z$-graded abelian groups
$$  KK_*(B,C)\otimes   KK_*(A,B)\to  KK_*(A,C)\ . $$
The isomorphism \eqref{fdsvsoihwjiovwefvwf} extends to an isomorphism of $\Z$-graded abelian group-valued functors 
\begin{equation}\label{asvoicjadsocadscadcwq}KK_*(\C, -) \cong K_*(-)
\end{equation}  (use \cite[Prop. 17.5.6]{blackadar}).
In particular we have an isomorphism $KK_*(\C,\C)\cong K_*(\C)\cong  \KU_*\cong \Z[\beta,\beta^{-1}]$.
Since $\kk_{\sepa}^0(\C)$ is the tensor unit
in $\KK_{\sepa}^0$, by specializing \eqref{afdasdfdscasdcasd} and \eqref{sdvsdcsacd} we may view the $\Z$-graded abelian groups $KK_*(A,B)$ as the morphisms groups in a $\KU_*$-enriched category $\KK_{\sepa}^*$. This expresses the compatibility of the composition in $\KK_\sepa^*$ with Bott periodicity.
In total, for $A,B$ in
$\nCalg_{\sepa}$ the $\Z$-graded abelian group $KK_*(A,B)$ has a structure of a $\KU_{*}$-module, and the composition product factorizes as 
 $$KK_*(B,C)\otimes_{\KU_{*}} KK_*(A,B)\to KK_*(A,C)\ .$$ \begin{cor}
We have a symmetric monoidal functor
$$\kk_{\sepa}^*\colon \nCalg_{\sepa}\to \KK_{\sepa}^*$$
whose target is additive and enriched in $\Mod(\KU_{*})$.
\end{cor}

\begin{theorem}\label{erkgjowergwrfwerfw}
There is a canonical isomorphism of lax symmetric monoidal functors  
$KK_*(\C,-) \cong K_*(-) \colon \nCalg_{\sepa} \to \Mod(\KU_*)$.
\end{theorem}
\begin{proof}
We must show that the isomorphism \eqref{asvoicjadsocadscadcwq} refines to an isomorphism of lax symmetric monoidal functors in a canonical way. In view of 
 \cref{wtijgowergferwgwref} we can replace the target of these functors by the category $\Ab^{\Z}$ of graded abelian groups.

Since $\kk_\sepa^*(\C)$ is the tensor unit of $\KK_\sepa^*$,  the functor $\Hom_{\KK_\sepa^*}(\kk_\sepa^*(\C),-)$ is the initial lax symmetric monoidal functor from  $\KK_\sepa^*$ to $\Ab^{\Z}$. 
Furthermore,  $K_*\colon \Calg_{\sepa} \to \Ab^\Z$ is naturally lax symmetric monoidal, see e.g.\ \cite[Prop.\ 1.1]{HG}. In view of the universal property of   $\kk_\sepa^*$  it has a factorization
$$\xymatrix{ \Calg_{\sepa}\ar[rr]^{K_{*}}\ar[dr]^{\kk^{*}_{\sepa}}&&\Ab^{\Z}\\&\KK^{*}_{\sepa}\ar@{..>}[ur]^{\K_{*}}&}\ ,$$
where $\K_* \colon \KK_\sepa^* \to \Ab^\Z$ is again lax symmetric monoidal. Consequently, there is a unique lax symmetric monoidal transformation $\Hom_{\KK_\sepa^*}(\kk_\sepa^*(\C),-) \to \K_*(-)$ of functors from $\KK^{*}_{\sepa}$ to $\Ab^{\Z}$. The restriction of this  transformation along the   symmetric monoidal
functor $\kk^{*}_\sepa$ is a lax symmetric monoidal transformation
$KK_*(\C,-)\to  K_*(-)$ refining the isomorphism 
\eqref{asvoicjadsocadscadcwq}.
\end{proof}

\subsection{Spectrum valued $KK$-theory} 
\label{wekojgwerferfwe1}

In this section we 
 turn to the homotopy theoretic aspects of $KK$-theory. 
 Often the structure of a triangulated category $\cT$ {is governed by} a stable $\infty$-category $\bT$ whose homotopy category $\ho(\bT)$ is equipped with a triangulated equivalence 
$\cT\simeq \ho(\bT)$.
For two objects $A,B$ in $\bT$ we then have the mapping spectrum
$\map_\bT(A,B)$ such that $$\pi_{n}\map_\bT(A,B)\cong \Hom_{\cT}(\ho(A),\Sigma^{-n}\ho(B))$$  for all $n$ in $\Z$.
 The triangulated category $\KK_{\sepa,0}$ is no exception to this  {principle}.
\begin{prop}[{\cite{LN}},\cite{Bunke:2023aa}]\label{werojigwergwerfw}
There exists an  idempotent-complete  stable $\infty$-category $\KK_{\sepa}$, a functor
$\kk_{\sepa}\colon\nCalg_{\sepa}\to \KK_{\sepa}$,  and a commutative triangle of $\infty$-categories
$$\xymatrix{\nCalg_{\sepa}\ar[dr]^{\kk_{\sepa}^0}\ar[rr]^{\kk_{\sepa}}&&\KK_{\sepa}\ar[dl]_{\ho}\\&\KK_{\sepa}^0&}$$
such that $\ho$ identifies   $\KK_{\sepa}^0$ with the homotopy category of $\KK_{\sepa}$ as triangulated categories.
\end{prop}
\begin{proof}({Idea})
A morphism $f\colon A\to B$  in $\nCalg_{\sepa}$ is called a $\kk_{\sepa}^0$-equivalence if $\kk_{\sepa}^0$ is an isomorphism. We let $W_{\kk_{\sepa}^0}$ denote the set of  $\kk_{\sepa}^0$-equivalences. Since the functor $\ho$ detects equivalences it is clear that 
  the functor $\kk_{\sepa }$ in   \cref{werojigwergwerfw}  (if it exists)  has the property that it sends
 $\kk_{\sepa}^0$-equivalences  
 to equivalences. As a candidate, we can therefore try the Dwyer-Kan localization  \begin{equation}\label{sdfvsfdvdfsvs} \kk_{\sepa}\colon\nCalg_{\sepa}\to \nCalg_{\sepa}[W_{\kk_{\sepa}^0}^{-1}]=:\KK_{\sepa}\end{equation}
Using the fibration category structure on $\nCalg_\sepa$ introduced in \cite{Uuye:2010aa} one can now verify that \eqref{sdfvsfdvdfsvs} indeed has the required properties. 
\end{proof}
 
The notions of homotopy invariance and $\mathbb{K}$-stability  from the previous section extend to functors with values in $\infty$-categories by replacing the words equal or isomorphism by equivalent or equivalences, respectively.

\begin{definition}A functor  from $\nCalg$ or $\nCalg_{\sepa}$ to a stable $\infty$-category is called exact (semi-exact, split-exact) if it sends the zero algebra to zero and exact (semi-exact, split-exact) sequences of $C^{*}$-algebras to fibre sequences.
  \end{definition}

\begin{prop}[\cite{LN}]\label{werojigwergwerfw1}
The functor
$\kk_{\sepa}\colon\nCalg_{\sepa}\to \KK_{\sepa}$ is 
homotopy invariant, $\mathbb{K}$-stable, and semi-exact. It is 
inital among functors from $\nCalg_{\sepa}$ to stable $\infty$-categories which are homotopy invariant, $\mathbb{K}$-stable, and semi-exact.
\end{prop}
\begin{proof}(Sketch)
 {First, we note that $\kk_\sepa$ is indeed homotopy invariant $\mathbb{K}$-stable and semi-exact.}
 {Now} suppose {that we are} given a homotopy invariant, $\mathbb{K}$-stable and semi-exact functor $F \colon \nCalg_{\sepa} \to \cC$, where $\cC$ is some stable $\infty$-category. {We need to show that $F$ factors essentially uniquely through an exact functor $\tilde{F} \colon \KK_\sepa \to \cC$.}
First, we argue that it factors through a (necessarily unique) functor $\KK_{\sepa} \to \cC$. Since the homotopy category $\ho(\cC)$ is additive and the canonical functor $\cC \to \ho(\cC)$ is conservative and split-exact, this follows from Higson's theorem \cref{weriojgweorgewrgfwerf9}. To see that the resulting functor $\KK_{\sepa} \to \cC$ is exact, one uses that any fibre sequence in $\KK_{\sepa}$ is the image of a semi-exact short exact sequence of $C^{*}$-algebras and the semi-exactness of $F$.
\end{proof}
 
\begin{rem}\label{sequence-of-localizations} {For later use, we} note that the localization map $\kk_{\sepa}\colon \nCalg_{\sepa}\to \KK_{\sepa}$ factors as a series of localizations:
\begin{equation}\label{gregwregfwfrfwrfwfwrfrf}\nCalg_{\sepa} \xrightarrow{L_{h}} \nCalg_{\sepa,h} \xrightarrow{L_{K}} L_{K}\nCalg_{\sepa,h} \xrightarrow{L_{se}} L_{K}\nCalg_{\sepa,h,se} \xrightarrow{\Omega^{2}} \KK_{\sepa} \ , 
\end{equation}
where $L_{h}$  formally inverts the homotopy equivalences,  $L_{K}$ inverts the left-upper corner inclusions $A \to A \otimes \mathbb{K}$, $L_{se}$  forces, by inverting the maps $\iota$ from \eqref{adscadscasdffccd}, semi-exact sequences to {be} fibre sequences, and  $\Omega^{2}$
 is a 
 right Bousfield localization
at the full subcategory of group objects.  Here $\nCalg_{\sepa,h}$ is left-exact with fibre sequences induced under $L_{h}$ by mapping cone sequences,
 $L_{K}\nCalg_{\sepa,h}$ and $L_{K}\nCalg_{\sepa,h,se}$ are in addition semi-additive, and
 $\KK_{\sepa}$ is stable. The functors $L_{K}$, $L_{se}$ and $\Omega^{2}$ are left-exact. The details are explained in  \cite{Bunke:2023aa}. In this reference {it is} further shown that $\nCalg_{\sepa,h}$ is simply the $\infty$-category associated to the topologically enriched category of separable $C^*$-algebras, and $L_{K}$ is a left Bousfield localization  at the $\mathbb{K}$-stable algebras.   All the above localization  functors have canonical symmetric monoidal refinements (also for the minimal tensor product).

 {Note that the above description of  $\KK_{\sepa}$ immediately implies \cref{werojigwergwerfw1} independently of the 
 classical literature.}
 \hB\end{rem}

\begin{prop}[{\cite{LN}, \cite{KKG}, \cite{Bunke:2023aa}}]\label{werojigwergwerfw2}
There is a canonical refinement of
$\KK_{\sepa}$ to a {stably} symmetric monoidal $\infty$-category such that
$\kk_{\sepa} $ refines to a symmetric monoidal functor.
The symmetric monoidal refinement of $\kk_{\sepa}$ is initial for lax symmetric monoidal functors from
$\nCalg_{\sepa}$ to symmetric monoidal stable $\infty$-categories which are homotopy invariant,  $\mathbb{K}$-stable and semi-exact functors.
 \end{prop}\begin{proof}(Sketch)
We use that $\otimes$ 
preserves $\kk_{\sepa}^0$-equivalences {in each variable} as a consequence of   \cref{weklporgwergfwerfw}. 
The proposition now  follows  from symmetric  monoidal  properties of  
Dwyer-Kan localizations \cite{hinich}.
\end{proof}

We again have a version of  \cref{werojigwergwerfw2} for the minimal tensor product.
 \begin{rem} \label{rtkgogfrewfrfw}
Let $\bF$ be any stable $\infty$-category and consider the full subcategory
$$\Fun^{\mathrm{hi,st,se}}(\nCalg_{\sepa},\bF)\subseteq \Fun (\nCalg_{\sepa},\bF)$$  consisting  of homotopy invariant, $\mathbb{K}$-stable, and semi-exact functors.   The universal property of $\kk_{\sepa}$ in   \cref{werojigwergwerfw1} says that  precomposition with $\kk_{\sepa}$ induces an 
equivalence
\begin{equation}\label{sdvdscadscasdcds}
\kk_{\sepa}^{*}\colon \Fun^{\mathrm{ex}}(\KK_{\sepa},\bF) \xrightarrow{\simeq}  \Fun^{\mathrm{hi,st,se}}(\nCalg_{\sepa},\bF)\ ,
\end{equation}  where  
$\Fun^{\mathrm{ex}} $ denotes the category of exact functors between stable $\infty$-categories.
Similarly, the refined universal property from   \cref{werojigwergwerfw2} states that  precomposition with the symmetric monoidal refinement of $\kk_{\sepa}$ induces  equivalences
\begin{equation}\label{sdvdscadscasdcds1}\kk_{\sepa}^{*}\colon\Fun^{ \mathrm{ex}}_{\otimes/\mathrm{lax}}(\KK_{\sepa},\bF) \stackrel{\simeq}{\to} \Fun_{\otimes/\mathrm{lax}}^{ \mathrm{hi,st,se}}(\nCalg_{\sepa},\bF)\end{equation} 
for every symmetric monoidal stable $\infty$-category,
where the subscript $\otimes/\mathrm{lax}$ indicates that we consider the categories of symmetric monoidal  or lax  symmetric monoidal functors. 
   \hB\end{rem}

We now extend $KK$-theory from separable to all $C^{*}$-algebras using $\Ind$-completions, that is, by formally adjoining filtered colimits, {see \cite[\S 5.5]{HTT}}. This construction provides a functor $\Ind\colon \mathrm{Cat}_{\infty}^{\mathrm{ex}} \to \mathrm{Pr}^{L}_{\mathrm{st}}$ from small stable $\infty$-categories and exact functors to presentable stable $\infty$-categories and colimit preserving (or equivalently left adjoint) functors.

  \begin{definition} We define  the presentable $\infty$-category
$$\KK:=\Ind(\KK_{\sepa}) \ .$$ 
\end{definition}
Since $\KK_{\sepa}$ is idempotent complete the canonical inclusion
\begin{equation}\label{fvsdfvvvsdvvfs}
y\colon \KK_{\sepa}\to \Ind(\KK_{\sepa})
\end{equation} 
identifies 
$\KK_{\sepa}$ with the full subcategory of $\KK$ {on} compact objects.
The $\Ind$-completion inherits the properties of being stable and the tensor structures, and $y$ is   exact and naturally  refines to a symmetric monoidal functor.
 \begin{definition}We define 
 { the} functor $\kk\colon \nCalg\to \KK$ by left-Kan extension 
$$\xymatrix{\nCalg_{\sepa}\ar[r]^{\kk_{\sepa }}\ar[dr]^{\incl} &
\KK_{\sepa}\ar[r]^{y}\ar@{==>}[d]&\KK\\&\nCalg \ar[ur]^{\kk}&} \ .$$
 \end{definition}
 
In order to state the universal property of $\kk$ in analogy to  \cref{werojigwergwerfw1} we introduce the following condition.

\begin{definition}\label{wjigriowrtgwefrefwerf}A functor $F$ on $\nCalg$ is called $s$-finitary, if
the canonical map
$$\colim_{A'\subseteq_{\sepa} A} F(A')\to F(A)$$ is an equivalence for all $A$ in $\nCalg$, where $A'$ in the index poset  of the  colimit runs over the separable subalgebras of $A$.
\end{definition}

\begin{cor}[{\cite{KKG},\cite{Bunke:2023aa}}]\label{ewkorgpegerwfrefw}
 The functor
$\kk\colon\nCalg \to \KK $ is homotopy invariant, $\mathbb{K}$-stable, semi-exact and $s$-finitary.
It is inital for functors from $\nCalg$ to presentable stable $\infty$-categories which are homotopy invariant, $\mathbb{K}$-stable,    semi-exact and $s$-finitary.
 \end{cor} 

 \begin{cor}[{\cite{KKG},\cite{Bunke:2023aa}}]\label{ewkorgpegerwfrefw1}
The presentable stable $\infty$-category $\KK$
refines to a presentably  symmetric  monoidal stable $\infty$-category and $\kk$ refines to a symmetric monoidal functor. 
This symmetric monoidal refinement is initial for lax symmetric monoidal functors from $\nCalg$ to
presentable symmetric monoidal stable $\infty$-categories which are homotopy invariant, $\mathbb{K}$-stable,    semi-exact and $s$-finitary.
 \end{cor} 
 
There is a version of \cref{ewkorgpegerwfrefw1} for the minimal tensor product.
 \begin{rem}
 Similarly as in  \cref{rtkgogfrewfrfw}
let $\bF$ be any presentable stable $\infty$-category and consider the full subcategory
$$\Fun^{\mathrm{hi,st,se,sfin}}(\nCalg ,\bF)\subseteq \Fun (\nCalg ,\bF)$$   consisting  of homotopy invariant, $\mathbb{K}$-stable, semi-exact,   and $s$-finitary functors.   The universal property of $\kk$ from \cref{ewkorgpegerwfrefw} says that the functor  
$$\kk^{*}\colon \Fun^{L}(\KK ,\bF) \to  \Fun^{\mathrm{hi,st,se,sfin}}(\nCalg ,\bF)$$ given by  
the precomposition with $\kk $  is a well-defined equivalence, where  
$\Fun^{L} $ denote the category of  left-adjoint ({or equivalently colimit preserving}) functors between presentable stable $\infty$-categories. 
Its symmetric monoidal version in   \cref{ewkorgpegerwfrefw1} says that 
$$\kk^{*}\colon \Fun^{L}_{\otimes/\mathrm{lax}}(\KK ,\bF) \to  \Fun_{\otimes/\mathrm{lax}}^{ \mathrm{hi,st,se,sfin}}(\nCalg ,\bF)$$
is an equivalence for every  presentably  symmetric  monoidal stable $\infty$-category $\bF$. 
  \hB\end{rem}

 \begin{proof}[Derivation of   \cref{ewkorgpegerwfrefw}]
We indicate how    \cref{ewkorgpegerwfrefw} is deduced from    \cref{werojigwergwerfw1}. 
Let $\bF$ be a presentable stable $\infty$-category. 
Consider the inclusion  \begin{equation}\label{sfdvsvvsffvsfdfff}
i\colon \nCalg_{\sepa}\to \nCalg\ .
\end{equation}  
Using the pointwise formula  for left Kan extensions we first observe that
 the restriction along $i$  induces 
 an equivalence  
 \begin{equation}\label{sdafasdfadfadsfadfadfa1} i^{*}\colon \Fun^{\mathrm{sfin}}(\nCalg,\bF)\stackrel{\simeq}{\to} \Fun(\nCalg_{\sepa},\bF)\end{equation} whose inverse is given by 
 left Kan-extension functor
 \begin{equation}\label{sdafasdfadfadsfadfadfa} i_{!}\colon \Fun(\nCalg_{\sepa},\bF)\stackrel{\simeq}{\to}  \Fun^{\mathrm{sfin}}(\nCalg,\bF)\ .\end{equation}  
 Further exploiting the pointwise formula we then observe that
  the equivalence  \eqref{sdafasdfadfadsfadfadfa} restricts 
 to an equivalence between   full functor subcategories
\begin{equation}\label{sdfvsdfvfvsvsfdvsv}
i^* \colon \Fun^{\mathrm{hi,st,se,sfin}} (\nCalg,\bF) \stackrel{\simeq}{\to} \Fun^{\mathrm{hi,st,se}} (\nCalg_\sepa,\bF)\ .
\end{equation} 
We now consider the square
 $$\xymatrix{ \Fun^{\mathrm{hi,st,se,sfin}} (\nCalg,\bF)  \ar[d]^{i^{*}}_{\simeq}&\ar[l]_-{\kk^{*} }\Fun^{L}(\KK,\bF)\ar[d]^{y^{*}}_{\simeq}\\\Fun^{\mathrm{hi,st,se}} (\nCalg_\sepa,\bF )&\ar[l]^-{\simeq}_-{\kk^{*}_{\sepa}}\Fun^{\mathrm{ex}}(\KK_{\sepa},\bF)}\ ,$$
 where the right vertical equivalence is given by the universal property of $y$ in \eqref{fvsdfvvvsdvvfs}, and the lower horizontal equivalence is given by   \cref{werojigwergwerfw1}, compare with \eqref{sdvdscadscasdcds}.
  We conclude that $\kk^{*}$ is an equivalence.
 \end{proof}
 
 \begin{proof}[Derivation of   \cref{ewkorgpegerwfrefw1}]
We now employ   \cref{werojigwergwerfw2} and argue similarly as in the derivation of  \cref{ewkorgpegerwfrefw}.
We use the general fact that if $\bA,\bB$ are symmetric monoidal $\infty$-categories such that
the Day convolution structure
on $\Fun(\bA,\bB)$ exists, then 
we have an equivalence
$$\Fun^{ }_{\mathrm{lax}}(\bA,\bB)\simeq \CAlg(\Fun(\bA,\bB)),$$
{see \cite[\S 2.2.6]{HA}.}
 
 Let $\bF$ be a presentable symmetric monoidal stable $\infty$-category.
 Since $i$ in \eqref{sfdvsvvsffvsfdfff}   is symmetric monoidal, we deduce that the left Kan extension functor  $ i_!$ in \eqref{sdafasdfadfadsfadfadfa} 
canonically refines  to a symmetric monoidal  functor. Consequently so does its inverse   $i^{*}$ from \eqref{sdafasdfadfadsfadfadfa1}.
The symmetric monoidal equivalence  
$i^{*}$ therefore induces an equivalence between the $\infty$-categories of commutative algebra objects in the functor categories, 
or equivalently, an equivalence
$$i^{*}\colon \Fun_{\otimes/\mathrm{lax}}^{\ \mathrm{sfin}}(\nCalg,\bF)\stackrel{\simeq}{\to}  \Fun^{ }_{\otimes/\mathrm{lax}} (\nCalg_{\sepa},\bF)\ .$$
We now use  the equivalence \eqref{sdfvsdfvfvsvsfdvsv}  in order to deduce the equivalence
$$i^{*} \colon \Fun^{ \mathrm{hi,st,se,sfin}}_{\otimes/\mathrm{lax}}(\nCalg,\bF)\stackrel{\simeq}{\to} \Fun^{ \mathrm{hi,st,se}}_{\otimes/\mathrm{lax}}(\nCalg_\sepa,\bF)\ .$$
The assertion of    \cref{ewkorgpegerwfrefw1}  follows from the decorated commutative square
 $$\xymatrix{ \Fun^{ \mathrm{hi,st,se,sfin}}_{\otimes/\mathrm{lax}} (\nCalg,\bF)  \ar[d]^{i^{*}}_{\simeq}&\ar[l]_-{\kk^{*} } \Fun^{\otimes,L}_{\otimes/\mathrm{lax}}(\KK,\bF)\ar[d]^{y^{*}}_{\simeq}\\\Fun^{ \mathrm{hi,st,se}}_{\otimes/\mathrm{lax}} (\nCalg_\sepa,\bF )&\ar[l]^-{\simeq}_-{\kk^{*}_{\sepa}}\Fun^{ \mathrm{ex}}_{\otimes/\mathrm{lax}}(\KK_{\sepa},\bF)}\ ,$$
 where the right vertical equivalence is given by a corresponding symmetric monoidal variant of the universal property of $y$ in \eqref{fvsdfvvvsdvvfs}, and the lower horizontal equivalence is given by  \cref{werojigwergwerfw2}, compare with \eqref{sdvdscadscasdcds1}. 
\end{proof}

 { {From now on we will use the notation
\[ \mathrm{KK}(-,-) := \map_{\KK}(-,-) \colon \KK^\op \times \KK \to \Sp, \quad  
 KK(-,-) :=  {(\kk^{\op}\times \kk)^*\mathrm{KK}(-,-)} \]  }

 Since $\kk(\C)$ is the tensor unit of the  stable symmetric monoidal $\infty$-category $\KK$ we get a commutative  algebra $\mathrm{KK}(\kk(\C),\kk(\C))$
in $\Sp$  which turns out to be equivalent to $\KU$.
   In view of   \cref{qrijgfqowefdwefqwefqwedqwed} below, this is compatible with our previous definition of $\KU$  as $K(\C)$.  
We define the functor  \begin{equation}\label{fjeiofjewrfredew}
{\K :=}\mathrm{KK}(\kk(\C),-)\colon \KK\to \Sp\ .
\end{equation}
It is the right-adjoint of
 {an adjunction}
\begin{equation}\label{eq:adjunction}
 - \otimes_{\KU} \kk(\C) \colon \Mod(\KU) \rightleftarrows \KK \colon  \K
\end{equation}
 {where the left adjoint is canonically symmetric monoidal, see e.g.\ \cite[Construction 5.23]{MNN}. As a consequence, $\KK$ is a commutative $\Mod(\KU)$-algebra in $\mathrm{Pr}^{L}_{\mathrm{st}}$. In particular, it is canonically endowed with the structure of a $\Mod(\KU)$-module in $\mathrm{Pr}^{L}_{\mathrm{st}}$.  Such presentable $\infty$-categories are called $\KU$-linear. 
  As a consequence of $\KU$-linearity, the mapping spectrum functor refines to a functor 
  $$\mathrm{KK}(-,-)\colon \KK^{\op}\otimes \KK \to \Mod(\KU)\ .$$
   }

By construction, for $A,B$ in $\nCalg_\sepa$, we obtain
 $$KK_{*}(A,B)\cong \pi_{*}KK(A,B)$$
in $\Mod(\KU_{*})$, i.e.\ the homotopy groups of the $KK$-spectrum $KK(A,B)$ recover the classical $KK$-groups of Kasparov.

\subsection{$K$-theory from $KK$-theory}\label{sec:3.3} In this section we  finish the construction of the spec\-trum-valued
 $K$-theory  functor  from \cref{wekojgwerferfwe}.
 {As the right adjoint of a symmetric monoidal functor,} the functor $\K$ from \eqref{fjeiofjewrfredew} 
refines to a lax symmetric monoidal functor
 $\KK\to \Mod(\KU)$.
 \begin{definition}\label{qrijgfqowefdwefqwefqwedqwed}
 We define the lax symmetric monoidal spectrum-valued $K$-theory functor as the composition 
 $$K:= \K\circ \kk  \colon \nCalg\to \Mod(\KU)\ .$$
 \end{definition}

\begin{cor}\label{wetkjgowergwerfw9} The functor $K$  is homotopy invariant, $\mathbb{K}$-stable,  semi-exact and $s$-finitary. Furthermore,  
we have a canonical isomorphism of  functors  $\pi_{*}  K \cong  K_{*}\colon\nCalg\to \Ab^{\Z}$.  \end{cor}
\begin{proof}
We use that $\kk$   is  homotopy invariant, $\mathbb{K}$-stable,  semi-exact and $s$-finitary.
We then conclude  the corresponding properties of $K$. In order to see semi-exactness we use  
 the general fact  that for any stable $\infty$-category $\cC$ and object $A$  the functor
$\map_{\cC}(A,-):\cC\to \Sp$ preserves limits. To see that $K$ is $s$-finitary we furthermore use that $\K$ (see \eqref{fjeiofjewrfredew}) preserves filtered colimits since
$\kk(\C)$ is a compact object of $\KK$.

We claim that $K_{*}$ and $\pi_{*}K$ are $s$-finitary. Then the desired isomorphism $\pi_{*}  K \cong  K_{*}$ is  induced via left Kan-extension by the corresponding isomorphism  given by  \cref{erkgjowergwrfwerfw} for the restrictions of these functors to $\nCalg_{\sepa}$.

Since $K$ is $s$-finitary as shown above and $\pi_{*}$ preserves filtered colimits we see that   $\pi_{*}  K$ is $s$-finitary. In order to see that
$K_{*}$ is $s$-finitray we use that 
 $K_{*}$ preserves filtered colimits and $\colim_{A'\subseteq A}A'\cong A$, where $A'$ runs over the separable subalgebras of $A$.  
\end{proof}

 {
\begin{rem}
As shown in   \cref{wekojgwerferfwe}, the functor
$K$ is actually exact, i.e., it sends all short exact sequences to fibre sequences. This 
property does not directly follow from the construction via $KK$-theory but rather uses the long-exact sequences of $K$-theory groups known 
from the classical approach. 
Alternatively one could use an $\infty$-categorical version of $E$-theory in order to construct $K$ as in \cite{Bunke:2023aa}.
\hB
\end{rem}
}

The following proposition shows that our model of spectrum-valued  $C^{*}$-algebra $K$-theory is canonically equivalent to any other possible model.

\begin{prop}\label{werjigoegsfg}
Let $F\colon\nCalg\to \Sp$ be a homotopy invariant, $\mathbb{K}$-stable, semi-exact and s-finitary functor.  {Then any isomorphism
$u_{0}\colon \pi_{0}K(-)\cong \pi_{0}F(-)$ on $\nCalg_{\sepa}$ canonically refines to an equivalence
$F\simeq K$ of functors. }
\end{prop}
\begin{proof}
We use the equivalence \begin{equation}\label{afadfadsfdsfa}
\kk^{*}\colon \Fun^{L}(\KK , {\Sp}) \stackrel{\simeq}{\to}  \Fun^{\mathrm{hi,st,se,sfin}}(\nCalg , {\Sp})
\end{equation}
given by \cref{ewkorgpegerwfrefw}.
We let $\tilde F$ in $\Fun^{L}(\KK ,\Sp) $ be a preimage under $\kk^{*}$ of $F$.
In view of the   \cref{qrijgfqowefdwefqwefqwedqwed} of $K$ 
and the Yoneda lemma we have an equivalence of  spaces
\begin{equation}\label{dacdscacadscda}
\Nat(K,F)\simeq \Nat(\K,\tilde F)\simeq  \Omega^{\infty}\tilde F(\kk(\C))\simeq \Omega^{\infty}F(\C)\ .
\end{equation}  The isomorphism $u_{0,\C}\colon \pi_{0}K(\C)\cong \pi_{0}F(\C)$ is determined by the image $u_{0,\C}(1_{\KU_{*}})$ in  $\pi_{0} \Omega^{\infty}F(\C)$.  The element $u_{0,\C}(1_{\KU_{*}})$ therefore determines
a natural transformation $\hat u\colon K\to F$ uniquely up to equivalence.
It remains to show that $\hat u$ is an equivalence inducing $u_{0}$ in zero homotopy on separable algebras.  
It suffices to show that 
 $\tilde u\colon \K \to \tilde F$  corresponding to $\hat u$ under the first equivalence in \eqref{dacdscacadscda} is an equivalence.    Since both functors are exact and {preserve filtered colimits, }
 it suffices to check that the induced map
 $\pi_{0}(\tilde u_{\kk(A)})\colon \pi_{0}\K(\kk(A))\to \pi_{0}\tilde F(\kk(A))$
 is an equivalence for every $A$ in $\nCalg_{\sepa}$.  But by construction,
 $\pi_{0}(\tilde u_{\kk(A)})\cong u_{0,A}$, and the latter is an isomorphism by assumption.
\end{proof}

\cref{qrijgfqowefdwefqwefqwedqwed} of the $K$-theory functor
  together with the universal property of $\kk$ formulated in   \cref{ewkorgpegerwfrefw} implies a characterization of the symmetric monoidal functor $\K$ by a universal property.

  \begin{prop}\label{gqerhgperfqrferfwrf}
  The $K$-theory functor $K\colon \nCalg\to \Sp$ is initial among lax symmetric monoidal functors from $\nCalg$ to  $\Sp$ which are
  homotopy invariant, $\mathbb{K}$-stable, semi-exact and $s$-finitary.
    \end{prop}
\begin{proof}
By Corollary \ref{ewkorgpegerwfrefw1} we have an equivalence
$$\kk^{*}\colon \Fun^{ L}_{\mathrm{lax}}(\KK,\Sp)\stackrel{\simeq}{\to} \Fun_{ \mathrm{lax}}^{ \mathrm{hi,st,se,sfin}}(\nCalg,\Sp)\ ,$$ and $K {=} \kk^{*} \K$ by Definition \ref{qrijgfqowefdwefqwefqwedqwed}.
Since the functor $\K$ is   corepresented by the tensor unit, it is initial in
$\Fun^{L}_{\mathrm{lax}}(\KK,\Sp)$ by 
\cite[Corollary 6.8 (3)]{Nikolaus}. 
Consequently, $K$ has the same property in  $\Fun_{\mathrm{lax}}^{  \mathrm{hi,st,se,sfin}}(\nCalg,\Sp)$.
  \end{proof}

  \section{Applications of spectrum-valued $KK$-theory}
In this section we illustrate  the usability of the stable $\infty$-categorical version of $\KK$ by giving several examples.

\subsection{Idempotents and units}\label{wrthijpwgergwergwrf} 
In this section we show that the usage of the homotopical version of $KK$ can simplify 
the discussion of unit spectra {for strongly selfabsorbing $C^{*}$-algebras} in  \cite{DP3}, \cite{DP1}, \cite{DP2}  considerably, {avoiding point-set constructions inside some model of spectra.}

We start with a general remark on monoid structures on mapping spaces.
  Recall that $\Mon(\cC)$ and $\CMon(\cC)$ denote the categories of monoids and commutative monoids in a category $\cC$ with cartesian products. In case $\cC=\Spc$ we omit the argument and simply write $\Mon$ and $\CMon$.
  For any object $C$ in an $\infty$-category $\cC$ the mapping space $\Map_{\cC}(C,C)$ has 
 a monoid structure given by composition.
 Let $(\cC,\otimes,\beins)$ be a symmetric monoidal $\infty$-category.
\begin{lem}\label{Lemma:monoids}\mbox{}
\begin{enumerate}
\item $\Map_\cC(\beins,\beins)$ is canonically an object of $\CMon$ with the (multiplication) structure induced from the commutative algebra and coalgebra structures on $\beins$.
\item The identity of $\Map_{\cC}(\beins,\beins)$ refines to an  equivalence of monoids between the multiplication structure and the composition structure.
\end{enumerate}

\end{lem}
\begin{proof}
We begin with part $(1)$. First, we record that there is an adjunction 
\begin{equation}\label{frewrfrewfwfwerf}M\mapsto \mathrm{B}M :\Mon    \leftrightarrows  \Cat_{*}: (\cC,C)\mapsto \Map_{\cC}(C,C) 
\end{equation} 
where the monoid structure on $\Map_{\cC}(C,C)$ refines the composition of endomorphisms.
Both functors preserve products and are therefore canonically symmetric monoidal with respect to the cartesian monoidal structures on both categories. 
Then we note that the forgetful functor 
\begin{equation}\label{qfwepfokqpwefqwedqewdq}
\CMon(\Cat_*) \to \CMon(\Cat)
\end{equation} is an equivalence, and the latter is the category of symmetric monoidal $\infty$-categories. Under the inverse of this equivalence, a symmetric monoidal category obtains a distinguished object given by the tensor unit of the monoidal structure.

 {The composition of the symmetric monoidal right-adjoint of the adjunction  \eqref{frewrfrewfwfwerf}  with the inverse of the equivalence in \eqref{qfwepfokqpwefqwedqewdq} induces a functor
\begin{equation}\label{rfojaspfsdfdfdfa} \CMon(\Cat) \lto \CMon(\Mon)\ .
\end{equation} 
Applying $\CMon(-)$ to the forgetful functor $\Mon \to \Spc$ gives a functor \begin{equation}\label{wefqewfewdq}\CMon(\Mon)\to \CMon
\end{equation}  which is in fact an equivalence. The composition of  \eqref{rfojaspfsdfdfdfa} with \eqref{wefqewfewdq}
sends a symmetric monoidal $\infty$-category $\cC$ to the commutative monoid $\Map_\cC(\beins,\beins)$ with multiplication structure showing (1). 

To see (2), we observe that, by construction, the forgetful (forgetting the commutative monoid structure) functor \begin{equation}\label{wrqefkjqwopfewfdqewqd}\CMon(\Mon) \to  \Mon\ ,
\end{equation}  composed with \eqref{rfojaspfsdfdfdfa} sends a symmetric monoidal category $\cC$ to the monoid $\Map_\cC(\beins,\beins)$ with composition structure.}

Now we use that the two forgetful functors 
\[\CMon(\Mon) \xrightarrow{\eqref{wrqefkjqwopfewfdqewqd}} \Mon \ , \quad \CMon(\Mon) \xrightarrow{\eqref{wefqewfewdq}} \CMon \to \Mon\]
are canonically equivalent. This is, for instance, because they are given by restriction along the two canonical maps of $\infty$-operads $\E_1 \to \E_1 \otimes \E_\infty$ which are canonically equivalent since $\E_1 \otimes \E_\infty\simeq \E_{\infty}$ is a terminal $\infty$-operad. 
\end{proof}

The following is the analog of \cref{Lemma:monoids} in the context of stable categories. For each object $C$ of   a stable $\infty$-category $\cC$  the spectrum $\map_{\cC}(C,C)$ refines to an object of $\Alg(\Sp)$ with the composition structure.
 
 We recall that a stably symmetric monoidal $\infty$-category is an object of $\CAlg(\Cat^\st)$, where $\Cat^\st$ is the $\infty$-category of stable $\infty$-categories, equipped with the Lurie tensor product as symmetric monoidal structure.
Let $(\cC,\otimes,\beins)$ be a stably symmetric monoidal $\infty$-category.
\begin{lem}\label{eneferferferfe}
 \mbox{}
\begin{enumerate}
\item  $\map_{\cC}(\beins,\beins)$ is canonically an object of $\CAlg(\Sp)$.
\item  The identity of $\map_{\cC}(\beins,\beins)$ refines to equivalence of algebras for the 
multiplication structure and the composition structure.
\end{enumerate}
\end{lem}
\begin{proof}
The strategy is the same as before. This time, we make use of the adjunction
\[  \Alg(\Sp) \rightleftarrows  (\Cat^\st)_{*}  \]
where the left adjoint  sends an algebra $R$ to $\Mod^{\mathrm{f}}(R)$, the small stable sub-category of $\Mod(R)$ generated by $R$, and the right-adjoint  sends $(\cC,C)$ to the mapping spectrum $\map_\cC(C,C)$. The left adjoint refines canonically to a symmetric monoidal functor (in fact, composed with the idempotent completion it is a symmetric monoidal equivalence onto the full subcategory of $(\Cat^\st)_{*}$ consisting of those pointed stable categories $(\cC,C)$ where $\cC$ is idempotent complete and $C$ is a generator of $\cC$). Consequently, the right adjoint refines canonically to a lax symmetric monoidal functor. We therefore again obtain a functor
\[  \CAlg(\Cat^\st) \lto \CAlg(\Alg(\Sp)).\]
The rest of the argument is the same as in the proof of \cref{Lemma:monoids}.
\end{proof}

Following \cite[\S 4.8.2]{HA},  {we make the following definition.
\begin{definition}
An idempotent object  in a symmetric monoidal $\infty$-category $\cC$ with tensor unit $\beins$  is a  map $\epsilon\colon \beins \to C$ such that 
  $\epsilon\otimes \id\colon C\simeq \beins \otimes C\to C\otimes C$ is an equivalence.  
\end{definition}
}
  The  inverse of this map is  the multiplication map of a canonical commutative algebra structure on $C$ with unit  {map given by} $\epsilon$. The functor $L_{C}:=-\otimes C\colon \cC\to \cC$   is 
   a Bousfield localization with unit transformation
$\id_{\bC}\xrightarrow{-\otimes \epsilon} -\otimes C$.

We now assume that $\cC$ is stable. Then, as discussed earlier, for any object $C$ of $\cC$, the spectrum $\map_{\cC}(C,C) $ naturally refines to an associative algebra in $\Sp$
with respect to composition. If $C$ is a commutative algebra, then
the spectrum $\map_{\cC}(\beins,C)$ canonically refines to a commutative algebra in $\Sp$. If the unit $\epsilon\colon \beins\to C$ is an idempotent, then we have an equivalence
  $$\epsilon^{*}\colon\map_{\cC}(C,C) \xrightarrow{\simeq} \map_{\cC}(\beins,C)\ ,$$
  of associative algebras which shows that $\map_{\cC}(C,C) $ refines to a commutative algebra.
  In order to see this we apply \cref{eneferferferfe} to the symmetric monoidal $\infty$-category $(L_{C}\cC,C,\otimes)$.

\begin{definition}
A unital $C^{*}$-algebra $A$ is called homotopically selfabsorbing if
the map $A\to A\otimes A$ given by $a\mapsto 1_{A}\otimes a$ is a homotopy equivalence.
Equivalently, we require that the unit map $L_{h}(\C)\to L_{h}(A)$  is an idempotent in the localization $ 
\nCalg_{\sepa,h}$  {discussed in} \cref{sequence-of-localizations}.
\end{definition}

Consequently, a homotopically selfabsorbing $C^{*}$-algebra canonically refines to
 a commutative algebra object in $\nCalg_{\sepa,h}$.
Note that according to the standing hypothesis of the present survey the notion of homotopically selfabsorbing is defined with respect to the  maximal tensor product. But in order to connect with the $C^{*}$-literature  in the following we use the analogue for the  minimal tensor product. In this case, examples of   homotopically selfabsorbing $C^{*}$-algebras are given by 
 asymptotically selfabsorbing, or more generally, by separable strongly selfabsorbing $C^{*}$-algebras, as we shall explain in the following
\begin{ex}\label{wtiuhgwergwerfwerf}
{We recall that an} asymptotically selfabsorbing $C^{*}$-algebra is a unital $C^{*}$-algebra 
such that there exists an isomorphism $\phi\colon A\xrightarrow{\cong} A\otimes_{\min} A$ and  
a strictly continuous map $u\colon (0,1]\to U(A\otimes_{\min} A)$ such that
$\lim_{t\to 0} 
u_{t} \phi u_{t}^{*}= (1_{A}\otimes \id_{A})$ \cite[Def.\ 1.1]{DW}. {In short this means that the right tensor embedding is asymptotically unitarily equivalent to an isomorphism.} The family of homomorphisms 
 $u_{t} \phi u_{t}^{*}$ extends by continuity to a homotopy from $(1_{A}\otimes \id_{A})$ to  an isomorphism.
An  asymptotically selfabsorbing $C^{*}$-algebra is therefore  homotopically selfabsorbing.
 
A strongly selfabsorbing $C^{*}$-algebra is a unital $C^{*}$-algebra 
such that there exists an isomorphism $\phi\colon A\xrightarrow{\cong} A\otimes_{\min} A$ and a sequence 
 $ (u_{n})_{n\in \nat}$ in $U(A\otimes_{\min} A)$ such that
$\lim_{n\to \infty} 
u_{n} \phi u_{n}^{*}= (1_{A}\otimes \id_{A})$ \cite[Def.\ 1.1]{DW}{, i.e.\ the right tensor embedding is approximately unitarily equivalent to an isomorphism.}
By  \cite[Thm.\ 2.2]{DW}  and  \cite[Rem.\ 3.3]{MR2784504} 
 any  separable strongly selfabsorbing $C^{*}$-algebra is asymptotically selfabsorbing, and hence homotopically selfabsorbing. 
{Therefore, the Cuntz algebra $\mathcal{O}_{\infty}$ and the Jiang--Su algebra $\mathcal{Z}$ are homotopically selfabsorbing.}
 \hB 
 \end{ex}
 
 There are also non-unital $C^{*}$-algebras which canonically give rise to idempotents in $\nCalg_{\sepa,h}$:
 \begin{ex} The left upper corner inclusion $e\colon\C\to \mathbb{K}$ induces an idempotent object in $\nCalg_{\sepa,h}$.
Indeed, the map $\mathbb{K}\to \mathbb{K}\otimes \mathbb{K}$ given by $k\mapsto e\otimes k$ is homotopic to an isomorphism {(see for example \cite[Lemma~2.4]{DP1}).}
 The Bousfield localization associated to this idempotent is the localization $L_{K}$ mentioned in \cref{sequence-of-localizations}.
\hB
\end{ex}

Recall from \cite[Cor.\ 3.7]{Bunke:2023aa} that the mapping spaces in $\nCalg_{h}$ are presented via $\ell\colon\Top\to \Spc$ from  \eqref{sfdgsfdgfsgsfdg} by the topological   spaces $\underline{\Hom}(A,B)$
of homomorphisms from $A$ to $B$ with the point-norm topology.
In view of  \eqref{gregwregfwfrfwrfwfwrfrf} and \eqref{fvsdfvvvsdvvfs} we have a symmetric monoidal functor
$$\kk_{h}\colon \nCalg_{\sepa,h}\xrightarrow{\Omega^{2}\circ L_{se}\circ L_{K}} \KK_{\sepa}\xrightarrow{y} \KK\ .$$
If $(\epsilon\colon \C\to A)$  is an idempotent in $ \nCalg_{\sepa,h}$, then 
$L_{h}(A)$ refines to a commutative algebra in $ \nCalg_{\sepa,h}
$. Since $\kk_{h}$ is symmetric monoidal  and $K(A)\simeq {\map_{\KK}}(\kk_{h}(\C),\kk_{h}(A))$ we get the following consequence. \begin{cor}
If $(\epsilon\colon \C\to A)$ is an idempotent in $\nCalg_{\sepa,h}$,  then
$K(A)$ has a canonical structure of a commutative algebra in $\Mod(\KU)$.
\end{cor}
By \cref{wtiuhgwergwerfwerf} this
 corollary in particular applies to   {separable strongly selfabsorbing   algebras}  in which case
it has previously  been  shown  by \cite{DP3} by an explicit construction of the ring spectrum $\K(A)$
{as a commutative symmetric ring spectrum.}

If $(\epsilon\colon\C\to A)$ is an idempotent in $\nCalg_{\sepa,h}$, then
$\ell\underline{\End}(A)$ becomes a commutative monoid {in $\Spc$} with respect to the composition
and we get a homomorphism 
 $$\ell\underline{\End}(A)\simeq \Map_{\nCalg_{\sepa,h}}(A,A)\xrightarrow{\kk_{h}}  \Omega^{\infty}KK(A,A)\xrightarrow{\epsilon^{*}} \Omega^{\infty}KK(\C,A)\simeq \Omega^{\infty}_{\otimes}K(A) \ ,$$
 where the monoid structure on the target is induced by the commutative algebra structure on $K(A)$ indicated by the subscript $\otimes$.
 The commutative subgroup of invertible components of $\Omega^{\infty}_{\otimes}K(A)$ is usually denoted by 
  $\mathrm{gl}_1(K(A))$\footnote{identifying commutative groups in $\Spc$ and connective spectra } \cite{mayunits} (see \cite{MR3252967} for an $\infty$-categorical version), and the underlying  {associative} group
   is  {denoted} $\GL_{1}(K(A))$.    Let $\ell\underline{\End}(A)^{\inv}$ denote the commutative group of invertible components in
  $\ell \underline{\End}(A)$.  \begin{cor} If $(\epsilon\colon \C\to A)$ is an idempotent in $\nCalg_{\sepa,h}$, then
   we have  a homomorphism of commutative groups
  $$\ell\underline{\End}(A)^{\inv}\to \mathrm{gl}_1(K(A))\ .$$
  \end{cor}
  Note that the space $\mathrm{B}\GL_{1}(K(A))$ classifies parametrized spectra with fibre $K(A)$, also called twists of $K(A)$.  
  {One of the goals of}
   \cite{DP3}, \cite{DP1}, \cite{DP2} is to realize 
  twists  by bundles of actual algebras,  see also \cref{wkogpwgerfwerfwf}. To this end one wants to  approximate
  $\GL_{1}(K(A))$ in terms of the actual automorphisms of $A$.
    Let $\underline{\Aut}(A)$ be the topological  subgroup of invertible endomorphisms of $A$. 
  \begin{cor} We have homomorphisms of  groups
  \begin{equation}\label{regerfferfrefwerf}\ell\underline{\Aut}(A)\to \ell\underline{\End}(A)^{\inv}\to   \GL_1(K(A))\ .
\end{equation}  
  \end{cor}
  Note that we can vary the algebra $A$ in its $\KK$-class in order to 
  get \eqref{regerfferfrefwerf} close to an equivalence. {For instance, the stabilisation of the infinite Cuntz algebra $\mathcal{O}_{\infty} \otimes \mathbb{K}$ is $KK$-equivalent to $\C$, and it was shown in \cite{DP3} that \eqref{regerfferfrefwerf} induces an equivalence 
  \[\ell\underline{\Aut}(\mathcal{O}_{\infty} \otimes \mathbb{K})\simeq  \GL_1(\KU).\]}

\subsection{{Refined} multiplicative structures}
{
In the previous \cref{wrthijpwgergwergwrf} we  observed that   homotopically selfabsorbing $C^{*}$-algebras  give rise to commutative algebra objects in $\nCalg_{h}$ and therefore also in $\KK$. For every $k$ in $\{1,2,\dots,\infty\}$ one can now consider the category of $\E_{k}$-algebras in $\KK$.
These categories interpolate between the categories of algebras (the case $k=1$) and commutative algebras ($k=\infty$) in $\KK$.   {To the best of our knowledge}  $\E_{k}$-algebras in  $\KK$  for $k$ different from $\infty$ have not been studied  {explicitly} yet. 

Note that $\E_{k}$-algebra structures are interesting as they give rise to power operations which can in principle be used, for instance, to identify differentials in spectral sequences.

We first discuss   $K$-theoretic obstructions against the existence of a such structure. It employs that the lax symmetric monoidal functor $\K$ defined in \eqref{fjeiofjewrfredew} 
 sends  $\E_{k}$-algebras in $\KK$ to $\E_{k}$-algebras in $\Sp$.
 In the following, we will use that the $\KK$-class of the finite Cuntz algebra $\mathcal{O}_{n+1}$ is given by $\beins/n$. While  the tensor unit $\beins\simeq \kk(\C)\simeq \kk(\cO_{\infty})$ supports an $\E_{\infty}$-algebra structure, the $\KK$-classes $\beins/n\simeq \kk(\cO_{n+1})$ of the  finite Cuntz algebras for integers $n\not=1$}   do not:
 
\begin{lem}
If $n\ge 2$, then the object $\beins/n$  does not support an $\E_{\infty}$-structure. 
\end{lem}
\begin{proof} Assume by contradiction that  $\beins/n$ supports an $\E_{\infty}$-structure. Then $\KU/n\simeq \K(\beins/n)$ would  support an $\E_{\infty}$-algebra structure which is known to be false. Indeed, an $\E_{\infty}$-structure on $\KU/n$ would induce one on the $p$-completion at any prime $p$. In particular, we may choose such a prime $p$ which divides $n$, so that the $p$-completion is a non-trivial $K(1)$-local $\E_{\infty}$-ring. Any such ring spectrum has a $\delta$-ring structure on $\pi_{0}$ \cite{Hopkins} and therefore the unit cannot be $p$-power torsion, see e.g.\ \cite[Lemma 1.5]{Bhatt}. This  contradicts the fact that $\pi_{0}$ of  the $p$-completion of $\KU/n$  is isomorphic to $\Z/p^{k}\Z$ for  some positive integer~$k$.
\end{proof}

In the following, for an abelian group $M$ or ring $R$  the symbols $HM$ or $HR$ denote the Eilenberg-MacLane spectrum on $M$ in $\Sp$ or of $R$ in $\Alg(\Sp)$.
Let $A$ be in $\KK$.
\begin{prop}\label{gjoiegjorefwefewrf}
Assume that $A$ admits an $\E_{2}$-algebra structure. If the unit in the ring  $\pi_{0}\K(A)$ is $p$-torsion for a prime $p$, then $\K(A) \simeq 0$.
\end{prop}
\begin{proof}
If $A$ has an $\E_{2}$-algebra structure, then $\K(A)$ is an $\E_{2}$-algebra in $\Sp$. Furthermore, by assumption $\pi_{0}\K(A)$ is an $\F_{p}$-algebra. By a result of Hopkins and Mahowald, see e.g.\ \cite[Theorem 5.1]{BAC}, it follows that $\K(A)$ is a  $H\F_{p}$-algebra and therefore a retract of $H\F_{p} \otimes \K(A)$. The latter is a module over $H\F_{p} \otimes \KU \simeq H\Z \otimes \KU/p$ and therefore trivial since the Bousfield class of $\KU/p$ is the same as that of $K(1)$ and $H\Z$ is $K(1)$-acyclic, see e.g.\ \cite[Lemma 2.2]{LMMT},  {so that $H\Z \otimes \KU/p=0$.}
\end{proof}

It follows that the  $\KK$-class   $\kk(\O_{p+1})\simeq \beins/p$
  of the finite Cuntz algebra with $p$ a prime does not even support an $\E_{2}$-algebra structure.
In fact, otherwise \cref{gjoiegjorefwefewrf} would imply that  $0\simeq \K(\beins/p)\simeq \KU/p$ which is false.

But the above does not rule out that $\mathcal{O}_{n+1}$ supports an $\E_{k}$-structure for $k\geq 2$ when $n$ is not prime. And indeed, the recent spectacular results in homotopy theory \cite{Burklund}  give a variety of positive examples. Let $n\geq 1$.

\begin{prop}\label{Prop:Ek-structures}
 {The object $\beins/2^{\lceil \tfrac{3}{2}(n+1) \rceil}$ and
for any odd prime $p$, the objects $\beins/p^{n+1}$ refine to $\E_{n}$-algebras in $\KK$.}
\end{prop}
\begin{proof}
{
Indeed, \cite[Theorem 1.2]{Burklund} says that the result is true in the symmetric monoidal $\infty$-category $\Sp$. Moreover, there is a (unique) symmetric monoidal and colimit preserving functor $\Sp \to \KK$. Since this functor sends
$\bS/k$ to $\beins/k$ for any integer $k$, the assertion follows.
}
\end{proof}

 {The above does not show that $\mathcal{O}_{n+1}$ gives rise to an $\E_{1}$-algebra for all $n$, simply because $\bS/n$ is not always $\E_{1}$, for instance not when $n=p$ is a prime. Nevertheless, we have the following unconditional positive result for integers $n$:
\begin{lem}\label{Lemma:E1-structures}
 {The object $\beins/n$ refines to an $\E_{1}$-algebra in $\KK$.}
\end{lem}
\begin{proof}
We consider the symmetric monoidal colimit preserving functor \[{-\otimes_{\KU}\kk(\C)\colon }\Mod(\KU) \to \KK\] from \eqref{eq:adjunction}. 
Since it sends $\KU/n$ to $\beins/n$, it suffices to show that $\KU/n$ refines to an $\E_{1}$-algebra in $\Mod(\KU)$. This is well-known to topologists and can for instance be seen as follows: Consider the element $n\beta$ in     $ \pi_{2}(\mathrm{GL}_{1}(\KU))$ and represent it by a map $S^{3} \to \mathrm{BGL}_{1}(\KU)$. The colimit of the composite of this map with the tautological inclusion $\mathrm{BGL}_{1}(\KU) \subseteq \Mod(\KU)$ is given by $\KU/n\beta \simeq \KU/n$, see e.g.\ \cite[Lemma 3.10]{Land-MJM}\footnote{In loc.\ cit.\ one can replace $\bS$ by $\KU$ throughout the argument.}. It then suffices to show that $S^{3} \to \mathrm{BGL}_{1}(\KU)$ can be refined to a map of $\E_{1}$-monoids in $\Spc$ \cite[Corollary 3.1]{BAC}, or
equivalently, that the induced map $S^{4} \to \mathrm{BBGL}_{1}(\KU)$ extends along bottom cell inclusion $S^{4} \to \mathrm{B}S^{3}$.
Since $\mathrm{B}S^{3}$ admits a cell structure with cells only in dimensions $4k$, the obstructions to extending the map $S^{4} \to \mathrm{BBGL}_{1}(\KU)$ to $\mathrm{B}S^{3}$ lie in the homotopy groups $\pi_{4k-1}(\mathrm{BBGL}_{1}(\KU))$ for $k\geq 2$. 
These homotopy groups are zero, since the odd homotopy groups of $\KU$ vanish, showing the lemma.  
\end{proof}

\subsection{Algebraic $K$-theory}
Let  $K^{\alg}\colon \Alg\to \Sp$ be the non-connective algebraic $K$-theory functor from (unital) associative algebras to spectra. {We note that it is canonically lax symmetric monoidal.} The restriction of $K^{\alg}$ along the {lax symmetric monoidal} forgetful functor $\Calg\to \Alg$ is known to be excisive in the sense that it sends pullback squares $$\xymatrix{A\ar[r]\ar[d] &B \ar[d] \\C \ar[r] &D }$$ in $\Calg$ whose vertical maps are surjective to cartesian squares in $\Sp$. Indeed, by a theorem of Suslin \cite[Thm.\ A]{suslin}, the functor $K^{\alg}$ sends such a square to a cartesian square provided the kernel of the map $B\to D$ is $\Tor$-unital, see also \cite{zbMATH06914269}, \cite{zbMATH07128142},  for new proofs of Suslin's excision result. We now use that $C^{*}$-algebras are $\Tor$-unital by Wodzicki \cite{wod}.  One can  {therefore} extend $K^{\alg}$ to non-unital $C^{*}$-algebras in the usual fashion by setting $K^{\alg}(A):=\Fib(K^{\alg}(A^{+})\to K^{\alg}(\C))$, where $A^{+}$ denotes the unitalization of $A$: It follows from excisiveness of $K^{\alg}$ that we have not changed the values on unital $C^{*}$-algebras. It also follows that $K^{\alg}$ defined in this way is lax symmetric monoidal, see e.g.\ \cite[Appendix A]{LN}. 

Let $B$ be any  $C^{*}$-algebra and consider the functors $$K_{B}, K^{\alg}_{B}\colon \nCalg\to \Sp \ ,\quad A\mapsto K(A\otimes_{\max} B )\ , \quad  A\mapsto K^{\alg}(A\otimes_{\max} B)\ .$$  {Both these functors are $s$-finitary. Indeed, this follows from the facts that $-\otimes_{\max} B$ preserves filtered colimits \cite[Lemma 7.13]{KKG}, that both topological and algebraic K-theory preserve filtered colimits, and that the forgetful functor from $C^{*}$-algebras to (non-unital) rings preserves the filtered colimit given by the family of separable subalgebras.}

We now assume that  $K^{\alg}_{B}$ is $\mathbb{K}$-stable. This is the case for instance when $B$ is properly infinite, see \cite[Prop.\ 2.2]{cotphi}, {such as
the Cuntz-algebras $\mathcal{O}_{n}$ for $2\leq n \leq \infty$ or the Jiang--Su algebra $\mathcal{Z}$, or when $B$ tensorially absorbs the compact operators $\mathbb{K}$}. Since $K^{\alg}_{B}$ is excisive by the discusssion above {(recall that the maximal tensor product preserves pullback squares whose vertical maps are surjections)}  it is semi-exact, and thus in particular split-exact. By a result of Higson \cite{higsa}, a $\mathbb{K}$-stable and split-exact functor is automatically homotopy invariant.
The functor
$K^{\alg}_{B}$ is thus homotopy invariant, $\mathbb{K}$-stable, semi-exact and $s$-finitary.
 By \cref{ewkorgpegerwfrefw} it has an essentially unique factorization
$$\xymatrix{\nCalg\ar[dr]_{\kk}\ar[rr]^{K^{\alg}_{B}}&&\Sp\\&\KK\ar@{.>}[ur]_{\K^{\alg}_{B}}&}\ .$$

{We now recall that there is a canonical isomorphism 
 \begin{equation}\label{iso-K-0}
 \pi_{0}K_{B}^{\alg}(-)\cong \pi_{0}K_{B}(-)
 \end{equation}
  of $\Ab$-valued functors 
 on $\nCalg$} \cite[Theorem 1.1]{Rosenberg-alg}.
 From \cref{werjigoegsfg}, we obtain the following version of  \cite[Thm.\ 3.2]{cotphi}. Let $B$ be a $C^{*}$-algebra.
\begin{cor}
If $B$ is {such that $K^{\alg}_{B}$ is $\mathbb{K}$-stable}, for instance a {properly} infinite $C^{*}$-algebra, then the isomorphism \eqref{iso-K-0} induces an equivalence 
of functors $K^{\alg}_{B}\simeq K_{B}$ from $\nCalg$ to $\Sp$.
\end{cor}
 
{The case $B=\mathbb{K}$ in the above corollary is Karoubi's conjecture \cite{zbMATH03693970},\cite{SW} asserting that $K^{\alg}_{\mathbb{K}}\simeq K$. Here we use that  the left upper corner inclusion $\C \to \mathbb{K}$ induces an equivalence $K_{\mathbb{K}} \simeq K$.
Similarly, since also $\mathcal{O}_{\infty}$ and $\mathcal{Z}$ are $\KK$-equivalent to $\C$, one obtains equivalences 
\[ K^{\alg}_{\mathcal{O}_{\infty}} \simeq K^{\alg}_{\mathcal{Z}} \simeq K.\]
Further examples of properly infinite $C^{*}$-algebras are the finite Cuntz algebras $\mathcal{O}_{n+1}$ for $n\geq 1$. {As already used earlier}, the unit $\C \to \mathcal{O}_{n+1}$ induces an equivalence in $\KK$ from $\kk(\C)/n$ to $\kk(\mathcal{O}_{n+1})$.
Consequently, $K_{\mathcal{O}_{n+1}} \simeq K/n$, and we obtain a canonical equivalence
\[ K^{\alg}_{\mathcal{O}_{n+1}} \simeq K/n.\]

 {
We record here that a conjecture of Rosenberg's \cite[Conjecture 3.1]{Rosenberg} asserts that for every $C^{*}$-algebra $A$ and positive natural number $n$, the canonical map 
\begin{equation}
\tau_{\geq0}[K^{\alg}(A)/n] \to \tau_{\geq 0}[K(A)/n]
\end{equation}
is an equivalence. This conjecture is known for certain classes of $C^*$-algebras, but is open in general. We refer to \cite{Rosenberg} for more details. Let us now consider the following commutative diagram}
\[\begin{tikzcd}
	\tau_{\geq0}[K^\alg(A)/n] \ar[r] \ar[d] & \tau_{\geq0}[K(A)/n] \ar[d,"\simeq"] \\
	\tau_{\geq0}[K^\alg(A \otimes \mathbb{K})/n] \ar[r,"\simeq"] & \tau_{\geq0}[K(A \otimes \mathbb{K})/n]
\end{tikzcd}\]
The right vertical map is an equivalence since $K$ is $\mathbb{K}$-stable and the lower horizontal map is an equivalence by the solution of Karoubi's conjecture. Therefore, Rosenberg's conjecture is equivalent to the statement that $\tau_{\geq0}K^\alg(-)/n$ is $\mathbb{K}$-stable. We note, of course, that the passage to connective covers is necessary as the example of $\C$ shows: $K^\alg(\C)/n$ is connective, whereas $K(\C)/n =\KU/n$ is periodic.}

 {Finally, we remark that by the above equivalence $K^\alg_{\mathcal{O}_{n+1}} \simeq K/n$, Rosenberg's conjecture is also equivalently phrased as asserting that the induced map
\[ K^{\alg}(A)/n \to K^{\alg}_{\mathcal{O}_{n+1}}(A) = K^{\alg}(A \otimes_{\max} \mathcal{O}_{n+1}) \]
is an equivalence on connective covers. Note that the existence of  the   map above  is not  obvious without the general discussion above.

  \subsection{$L$-theory}
 We use the material from \cite{LN}.
{Since every $C^{*}$-algebra has an underlying ring with involution, and algebraic $L$-theory is a lax symmetric monoidal functor from involutive rings to spectra, we obtain a {lax symmetric monoidal} $L$-theory functor $$L\colon\nCalg\to \Sp$$ for $C^{*}$-algebras. To be precise, we work with projective symmetric $L$-theory in order to obtain a well-defined extension to non-unital algebras. For unital $C^{*}$-algebras, one can also consider free symmetric $L$-theory, see \cite[Section 5]{LNS} for some results in this case.} For   $A$ in $\nCalg$ 
there exists a natural isomorphism of $\Z$-graded groups $$\tau_{A,*}\colon \pi_{*}K(A)\xrightarrow{\cong} \pi_{*}L(A)\ ,$$  
 {see \cite{LN}, \cite{LNS} for detailed discussions (in particular also for real $C^{*}$-algebras). {Note that the above isomorphism holds true only for complex $C^{*}$-algebras, but the isomorphism $\tau_{A,0}$ remains valid for real $C^{*}$-algebras and is induced by the observation that any finitely generated projective $A$-module is essentially uniquely a Hilbert-$A$-module, i.e.\ is equipped with a canonical positive definite form}. However, even for complex $C^{*}$-algebras, the isomorphism $\tau_{A,0}$ does not induce an equivalence  {of spectra} between $K(A)$ and $L(A)$. In fact, {even the spectra $\KU$ and $L(\C)$ are not equivalent, and do not even admit any non-trivial maps between each other, see \cite[Theorem E]{LN}}. We note that \cref{werjigoegsfg} is not applicable since $L$ is not semi-exact ({though it is split-exact}).

 Since $\kk\colon \nCalg_{\sepa}\to \KK_{\sepa}$ presents the latter as the Dwyer-Kan localization of 
  $\nCalg_{\sepa}$ at the $\kk_{\sepa}^0$-equivalences  and the restriction  $L_{\sepa}$  of $L$ to separable $C^{*}$-algebras is known to
  send $\kk_{\sepa}^0$-equivalences to equivalences  we have a factorization
  $$\xymatrix{\nCalg_{\sepa}\ar[dr]_{\kk_{\sepa}}\ar[rr]^{L_{\sepa}}&&\Sp\\&\KK_{{\sepa}}\ar@{..>}[ur]_{\hat L_{\sepa}}&}\ .$$
  But in contrast to the functors characterized in \cref{werojigwergwerfw1} the induced functor $\hat L_{\sepa}$
is not exact, i.e.\ does not belong to $\Fun^{\mathrm{ex}}(\KK_{\sepa},\Sp)$, see also \eqref{sdvdscadscasdcds}. {However, being split-exact, it does preserve products and so one can use similar techniques to obtain a canonical lax symmetric monoidal transformation from  the connective cover   $k=\tau_{\geq0}K$ of the K-theory functor to $L$. The main result of \cite{LNS} says that the induced map 
\[ k(A) \otimes_{\mathrm{ko}} L(\R) \to L(A) \]
is an equivalence for every real $C^{*}$-algebra, and consequently, that for each complex $C^{*}$-algebra $B$, a canonical map
\[ k(B) \otimes_{\mathrm{ku}} L(\C) \to L(B) \]
is an equivalence. From these equivalences of spectra, one can describe all $L$-groups (in particular of real $C^{*}$-algebras) in terms of $K$-groups and the action of a certain degree one operator on them (the multiplication by $\eta \in \KO_{1}(\R)$), see \cite[Theorem B]{LNS}.}

\subsection{{Universal coefficient classes}}
 \label{wotejgowferfwefwe}

In this example we explain how the homotopical point of view on $\KK$ simplifies the discussion of the universal coefficient theorem of Rosenberg--Schochet \cite{Rosenberg_1987}. Consider the functors
{$\K\colon \KK \to \Mod(\KU)$ from \eqref{fjeiofjewrfredew} and $ \K_{*}:=\pi_{*}\K\colon \KK\to \Mod(\KU_{*})$.  
 The functor $\K_{*}$ induces for all $A,B$ in $\KK$ a binatural map
\begin{equation}\label{wergergewfwerf}
\pi_{*}\mathrm{KK}(A,B)\to  \Hom_{\KU_{*}}(\K_{*}(A),\K_{*}(B))\ .
\end{equation} 
which is an isomorphism if $A=\kk(\C)$. But since the functor  
$ \Hom_{\KU_{*}}(-,\K_{*}(B))$  is not exact in general, we can not expect that 
\eqref{wergergewfwerf} is an isomorphism for all $A,B$.

As in {the K\"unneth problem discussed in}  \cref{rgjhiergewrg9}, it turns out to be useful to consider the homotopical version 
 \begin{equation}\label{asdcascdasxasxa}
\mathrm{KK}(A,B)\to \map_{\KU}(\K(A),\K(B))
\end{equation}
of  \eqref{wergergewfwerf} induced by $\K$.
Since $\KU_{*}$ has homological dimension $1$ we have, for all $A,B$ in $\KK$, an exact sequence
\begin{equation}\label{efqdadsdsasd}
0\to \Ext_{\KU_{*}}(\K_{*}(A),\K_{*{+1}}(B))\to \pi_{*}\map_{\KU}(\K(A),\K(B)) \to \Hom_{\KU_{*}}(\K_{*}(A),\K_{*}(B))\to 0\ .
\end{equation}
Furthermore we have a complex (not necessarily exact) 
\begin{equation}\label{efqdadsdsasd1}
0\to \Ext_{\KU_{*}}(\K_{*}(A),\K_{*+1}(B))\to \pi_{*}\mathrm{KK}(A,B) \to \Hom_{\KU_{*}}(\K_{*}(A),\K_{*}(B))\to 0\ .
\end{equation} 
{Here, the second map is induced by \eqref{asdcascdasxasxa} and the second map in \eqref{efqdadsdsasd}} and the first map lifts the first map from \eqref{efqdadsdsasd} along the map obtained from \eqref{asdcascdasxasxa} by applying $\pi_{*}$. In order to construct {this map} one chooses  a morphism $F\to A$ in $\KK$
such that  $\K_{*}(F)\to \K_{*}(A)$ is surjective and $F$ is an appropriate sum of copies of $\kk(\C)$ and $\Sigma \kk(\C)$. One then inserts the fibre sequence $\Fib(F\to A)\to F\to A$ into $\pi_{*}\mathrm{KK}(-,B)$  and finds the map as induced from the resulting long exact sequence.

{As a consequence}, if $A,B$  are  in $\KK$,  then the map \eqref{asdcascdasxasxa}
  is an equivalence if and only if the sequence  \eqref{efqdadsdsasd1}
is exact.  This relates the question  whether  \eqref{asdcascdasxasxa} is an equivalence {to} the classical universal coefficent theorem for $\KK$  \cite{Rosenberg_1987}.
\begin{definition}  \label{hwrthgwregregwerg}  
We let $\UCT$  be  the full subcategory of $\KK$  on objects $A$ such that  the map
 \eqref{asdcascdasxasxa} is an equivalence for all $B$ in $\KK$. 
 \end{definition}

 \begin{prop}\label{wklgpwrfgerwfw9}\mbox{}
  \begin{enumerate}
\item\label{tgijwoiergewrgwefw} The adjunction  \eqref{eq:adjunction}   is a right Bousfield localization.
\item\label{wergwoeigjowefewrf} The  composition
\[\K_{|\UCT} \colon \UCT\subseteq \KK\to \Mod(\KU)\]  
is an equivalence of stable $\infty$-categories.
 \item\label{ogfrerferijoetgweferw}  $\UCT$ is the localizing subcategory of $\KK$ generated by $\C$.\footnote{That is, $\UCT$ is the smallesst stable subcategory which is closed under colimits and contains $\kk(\C)$.}
  \end{enumerate}
\end{prop}
\begin{proof} We start with assertion \eqref{tgijwoiergewrgwefw}.
To this end we must show that
 the unit map $M \to \K(M \otimes_{\KU} \kk(\C))$ is an equivalence.
 It is obviously so for $M=\KU$. We now argue that its source and target preserve colimits and use that $\Mod(\KU)$ is generated under colimits by $\KU$. Indeed, the source is the identity, and for the target we employ the facts that $ - \otimes_{\KU} \kk(\C)$ preserves colimits, and $\K=\map_{\KK}(\kk(\C),-)$ also preserves colimits since $\kk(\C)$ is a compact object of $\KK$. 
The left adjoint  $- \otimes_{\KU} \kk(\C)$ of the adjunction \eqref{eq:adjunction} is therefore fully faithful.

In order to show assertion \eqref{wergwoeigjowefewrf} it suffices to show that the essential image of the 
left-adjoint  $-\otimes_{\KU}\kk(\C)$ in \eqref{eq:adjunction}
 is  $\UCT$. {By assertion \eqref{tgijwoiergewrgwefw}}, this essential image 
 consists of all objects $A$ of $\KK$ for which the 
counit $\K(A) \otimes_{\KU} \kk(\C) \to A$ of the adjunction  \eqref{eq:adjunction} 
is an equivalence. This counit induces the first map in the factorization
 \[ \mathrm{KK}(A,B) \to \mathrm{KK}(\K(A)\otimes_{\KU} \kk(\C),B) \stackrel{\eqref{eq:adjunction} }{\simeq} \map_{\KU}(\K(A),\K(B))\]
 of \eqref{asdcascdasxasxa}. In view of \cref{hwrthgwregregwerg}
  the object $A$ is in the essential image of  $- \otimes_{\KU} \kk(\C)$  if and only if it belongs to $\UCT$.

 In order to see assertion \eqref{ogfrerferijoetgweferw} we  use that
 $\Mod(\KU)$ is generated under colimits by $\KU$. In view of the equivalence from assertion \eqref{wergwoeigjowefewrf} the category $\UCT$ is then generated under colimits by $\kk(\C)$.
 \end{proof}

 \begin{cor}\label{qirjoqrfewqedew}\mbox{}
 \begin{enumerate}
 \item  \label{igjweorijfwerfrewf}
 For $A$ in $\UCT$ and $B$ in $\KK$, {the canonical map} 
\[ \K(A)\otimes_{\KU}\K(B){\to} \K(A\otimes B)\ \]
is an equivalence.
\item \label{hregiuwerhgfwerfwrefwref}The equivalence
  $$\K_{|\UCT} \colon \UCT \stackrel{\simeq}{\to} \Mod(\KU)\ .$$  
  from \cref{wklgpwrfgerwfw9} \eqref{wergwoeigjowefewrf} is symmetric monoidal.\end{enumerate}
\end{cor}
\begin{proof}
For \eqref{igjweorijfwerfrewf}, we fix $B$ and view the source and target of the map in question 
 as functors on $A$.  It is an equivalence for $A=\kk(\C)$.
 Since both  functors preserve colimits, the claim follows from  \cref{wklgpwrfgerwfw9} (\ref{ogfrerferijoetgweferw}) stating that 
 $\UCT$  is  generated under colimits by $\kk(\C)$.

Assertion \eqref{hregiuwerhgfwerfwrefwref} is a consequence of assertion \eqref{igjweorijfwerfrewf}
and  \cref{wklgpwrfgerwfw9} (\ref{wergwoeigjowefewrf}).
\end{proof}

The class of $C^{*}$-algebras 
$$\{A\in \nCalg\mid \kk(A)\in \UCT\}$$ contains $\C$, is closed under finite sums  {(since the functor $\kk \colon \nCalg \to \KK$ preserves finite sums)}  and satisfies the two-out-of three property for semi-exact sequences.

We  point out that the classical UCT-class is also closed under countable sums and filtered colimits. But this is not immediate from the present perspective since
$y\colon \KK_{\sepa} \to \KK$ is not compatible with infinite colimits. On the other hand, the test objects $B$ for $\UCT$ {according to our definition} are all objects of $\KK$, while in the classical case we only consider objects of $\nCalg_{\sepa}$.
In order to capture the classical $\UCT$  we consider the K-theory functor restricted to separable objects
$$\K_{\sepa}:=\K\circ y\markus{\colon}\KK_{\sepa} \to \Mod(\KU)$$
and the map
 \begin{equation}\label{asdcascdasxasxa1}
\mathrm{KK}_{\sepa}(A,B)\to \map_{\KU}(\K_{\sepa}(A),\K_{\sepa}(B))\ .
\end{equation}
\begin{definition} We let 
  $\UCT_{\sepa}$ be the  full subcategory of $\KK_{\sepa}$ of objects $A$  such that the map
  \eqref{asdcascdasxasxa1}
is an equivalence for all $B$ in $\KK_{\sepa}$.
\end{definition}

\begin{definition}
We let  $\Mod(\KU)_\sepa$ be full subcategory of $\Mod(\KU)$ on $\KU$-modules whose homotopy groups are countable.
\end{definition}
We note that $\Mod(\KU)_\sepa$ is the smallest stable subcategory of $\Mod(\KU)$ which is closed under countable colimits and contains $\KU$.

We then have the following  {result}, see also \cite[Thm.\ 6.1]{MR2193334}.
\begin{prop}\label{oktrhoeprgrtgertge}\mbox{}
\begin{enumerate}
 \item \label{gwkjorfgwerfwerf}The functor $\K_{\sepa}$ is the right-adjoint of a right Bousfield localization
$$- \otimes_\KU \kk_\sepa(\C) \colon \Mod(\KU)_\sepa \leftrightarrows  \KK_{\sepa}  \colon \K_{\sepa} $$ 
whose left-adjoint is essentially uniquely characterized by {sending $\KU$ to $\kk_\sepa(\C)$}.  \item\label{ogijoetgweferw11}  The  composition
\begin{equation}\label{ascasdcascdas}{\K_{\sepa}}_{|\UCT_\sepa} \colon \UCT_{\sepa}\subseteq \KK_{\sepa}\to \Mod(\KU)_\sepa \end{equation}   is an equivalence of stable $\infty$-categories.
\item\label{hfiuhuidfaffasffd}  $\UCT_{\sepa}$ is the countably-localizing subcategory
of $\KK_{\sepa}$ generated by $\kk_{\sepa}(\C)$.\footnote{That is, $\UCT_\sepa$ is the smallest stable subcategory which is closed under countable colimits and contains $\kk_\sepa(\C)$.}
\end{enumerate}\end{prop}
\begin{proof} We start with assertion \eqref{gwkjorfgwerfwerf}.
It is a classical fact  that the $K$-groups of a separable $C^{*}$-algebra are countable, so the functor $\K_\sepa$ as displayed is well-defined. 
 {We now show that $\K_\sepa$ admits a fully faithful left adjoint. To do so, it suffices to construct for any $M \in \Mod(\KU)_\sepa$ an object $M\otimes_\KU \kk_\sepa(\C)$ together with an equivalence $M \to \K_\sepa(M \otimes_\KU \kk_\sepa(\C))$ such that induced map 
\[ \map_{\KK_\sepa}(M\otimes_\KU \kk(\C),A) \to \map_{\KU}(M,\K_\sepa(A)) \]
is an equivalence for all objects $A$ of $\KK_\sepa$. Since every object $M$ of $\Mod(\KU)_\sepa$ is a cofibre of a map between sums of shifts of $\KU$ (using that $\KU_*$ has homological dimension 1), and $- \otimes_\KU \kk_\sepa(\C)$ preserves colimits, it suffices to construct these data for $\KU$, in which case we set $\KU \otimes_\KU \kk_\sepa(\C) = \kk_\sepa(\C)$ and define the wanted map to be the identity.}

The arguments for the remaining two assertions \eqref{ogijoetgweferw11} and  \eqref{hfiuhuidfaffasffd} are analoguos to the arguments for    \cref{wklgpwrfgerwfw9} (\ref{wergwoeigjowefewrf}) and 
\cref{wklgpwrfgerwfw9} (\ref{ogfrerferijoetgweferw}).
\end{proof}

\cref{oktrhoeprgrtgertge} (\ref{hfiuhuidfaffasffd})
says that $$\{A\in \nCalg_{\sepa}\mid \kk_{\sepa}(A)\in  \UCT_{\sepa}\}$$ is the bootstrap class in the sense of Rosenberg--Schochet \cite{Rosenberg_1987}.
 {A big open problem about the bootstrap class is whether it contains   all separable nuclear $C^*$-algebras. Our perspective has nothing to contribute to this question.}
 
\begin{rem}
 {A further open problem, which we address only briefly, is to describe 
 the Verdier quotient $\KK_\sepa/\UCT_\sepa$. 
 First, we note that this is a non-trivial stable $\infty$-category: Indeed, $\KK_\sepa/\UCT_\sepa$ being trivial is equivalent to the statement that a map $B \to C$ between separable $C^*$-algebras which induces an equivalence on $K$-theory in fact induces an equivalence in $\KK_\sepa$. This is not the case: Pick any short exact sequence \eqref{ergwergerferffw} of $C^*$-algebras. By \cref{wekojgwerferfwe}, $K$-theory sends all short exact sequences to fibre sequences, so the 
map $\iota\colon \ker(f) \to C(f)$ as in \cref{adscadscasdffccd} induces an equivalence on $K$-theory. However, this map is an equivalence in $\KK_\sepa$ if and only if the original short exact sequence induces a fibre sequence in $\KK_\sepa$, but there are short exact sequences whose image in $\KK_\sepa$ are not fibre sequences, see \cref{egojwpegerferfw}.} 
 
Basic invariants of small stable $\infty$-categories are so-called localizing (or additive) invariants. By construction, these depend only on the noncommutative motive associated to a stable $\infty$-category \cite{BGT}. Now, the noncommutative motives of $\KK_\sepa$ and $\UCT_\sepa$ both vanish since the existence of countable colimits allows for an Eilenberg--Swindle showing that the identity functor induces the zero map
 on the respective noncommutative motives. Moreover, the noncommutative motive of the Verdier quotient $\KK_\sepa/\UCT_\sepa$ is the cofibre of the induced map of noncommutative motives between $\UCT_\sepa$ and $\KK_\sepa$ and hence vanishes as well. 
Therefore typical invariants of small stable $\infty$-categories like algebraic $K$-theory or any flavour of topological Hochschild homology vanish on $\KK_\sepa/\UCT_\sepa$.
 \end{rem}

\subsection{Swan's theorem in $\KK$}  \label{adsfqdewdfew}
The following is a continuation of   \cref{wegojwoergferfweferfwr}.
We take advantage of the fact that we can do homotopy theory in $\KK$.
We use the notation from   \cref{eiorghbergerg}.

Let $A$ be in $\nCalg$  and  $X$ be a compact Hausdorff  space. 
\begin{prop}\label{wergijowegwergwrf}
If  $X$ is homotopy  finitely dominated, then we have an natural equivalence $\kk(C(X,A))\simeq \kk(A)^{X}$ in $\CAlg(\KK)$.
\end{prop}
\begin{proof}
The proof  is analogous to the proof of  \cref{weokjgopwegfrefwerf}.
Both  functors $X\mapsto \kk(C(X,A))$ and $X\mapsto \kk(A)^{X}$ are  homotopy invariant, 
excisive for cofibrant closed decompositions    and coincide for $X=*$. In the case of ${\kk(C(-,A))}$ we use 
that the exact sequences \eqref{fdvdvfsvdfv} and \eqref{fdvdvfsvdfv1} are semi-split. 
\end{proof}

Let $X$ be a compact Hausdorff space and $A,B$ in $\nCalg$.
\begin{cor}\label{erkogfwrewf}
If  $X$ is homotopy  finitely dominated, then
we  have natural equivalences
$$KK(C(X,A),B)\simeq KK(A,B)\otimes X\ , \quad KK(A,C(X,B))\simeq KK(A,B)^{X}\ $$
 {of $\KU$-modules. }
\end{cor}
\begin{proof} This is a formal consequence of \cref{wergijowegwergwrf} using that $\Sigma^{\infty}_{+}X$ is a dualizable object of $\Sp$. Alternatively one can 
argue as in the proof of said proposition.
 \end{proof}

\cref{erkogfwrewf} allows us to calculate the groups
$KK_{*}(C(X,A),B)$ and $KK_{*}(A,C(X,B))$ in terms of the Atiyah-Hirzebruch spectral sequences \eqref{casdcsdaccsca1} and \eqref{casdcsdaccsca} which only depend on {the spectrum}  $\KK(A,B)$   and the  {space} underlying $X$ -- in fact only the $\KU$-module $\KU\otimes X$.

The restriction to spaces which are homotopy {finitely dominated}   is again necessary.  The counterexample  from   \cref{wegojwoergferfweferfwr} {also applies in the present situation}.

 \subsection{Twisted $K$-theory}\label{wkogpwgerfwerfwf}
 
 Let $A$ be in $\nCalg$ and $X$ be a compact 
 Hausdorff space.
 Assume that $\cA\to X$ is a locally trivial fibre bundle of $C^{*}$-algebras with fibre $A$. The  transition functions between local trivializations are required to be continuous 
 in the  point-norm topology on the automorphisms of $A$.

We can form the $C^{*}$-algebra 
 $\Gamma(X,\cA)$ of continuous sections of $\cA$.  {In case of a trivial bundle, $\cA = X \times A$, we obtain $\Gamma(X,\cA) = C(X,A)$. We aim to generalize \cref{wergijowegwergwrf} to the case of non-trivial bundles}.
 \begin{definition}The $\cA$-twisted $K$-theory of $X$ is defined as the object
$K(\Gamma(X,\cA))$ in $\Mod(\KU)$.  \end{definition}

In this section we show how  $K(\Gamma(X,\cA))$ can be described as a twisted $K$-theory in the homotopy theoretic sense. The rough idea is as follows. We apply $K$ fibrewise to the bundle $\cA$ and get a  {``bundle''} $K(\cA)\to X$ of objects in $\Mod(\KU)$ with fibre $K(A)$.   This bundle is a twist in the sense of homotopy theory.  
The associated twisted  cohomology is then defined as the  $\KU$-module  $\Gamma(X,K(\cA))$ of sections of this bundle. Constructions of this type  have first been made precise in \cite{MR3252967} using the language of $\infty$-categories.
We then claim that \begin{equation}\label{svsdcacascca}
K(\Gamma(X,\cA))\simeq \Gamma(X,K(\cA))
\end{equation}
in $\Mod(\KU)$  if $X$ is  homotopy finitely dominated. 
The  precise definition of the right-hand side of this equivalence requires explanations  {which we provide now}.

Let $C$ be an object in an $\infty$-category $\cC$.
In a first step we make precise what we mean by a
``bundle'' ({henceforth called local system}) with fibre $C$  over a space $W$ in $\Spc$. To this end, we let $ {\bB\Aut_{\cC}(C)}$ denote the  full subcategory  {of $\cC^\simeq$ (the groupoid core of $\cC$) }on objects equivalent to $C$.
\begin{definition}
The category of  {local systems}
on $W$ with fibre  $C$ is the functor category $\Fun(W, {\bB\Aut_{\cC}(C)})$.  
\end{definition}

\begin{rem} 
We have a  {fully faithful} functor $\Spc\to \Cat_{\infty}$ identifying $\infty$-groupoids with spaces.  {This functor has been used in the above definition, and there is an equivalence of $\infty$-groupoids }
\[\Map_\Spc(W,\bB\Aut_\cC(C)) \simeq \Fun(W,\bB\Aut_\cC(C)).\] 
Moreover, a group $G$ in $\Group(\Spc)$  gives rise to a classifying space $BG$ whose associated $\infty$-category  has a single equivalence class of objects with automorphism group $G$.

The invertible components in the monoid $\End_{\cC}(C)$ form a group $\Aut_{\cC}(C)$ and we can consider its classifying space $B\Aut_{\cC}(C)$ in $\Spc$. The canonical inclusion $B\Aut_{\cC}(C)\to \cC$ 
 {induces an equivalence}
$B\Aut_{\cC}(C)\stackrel{\simeq}{\to}{\bB\Aut_{\cC}(C)}$.
Thus we have an equivalence of $\infty$-groupoids  
$$\Map_\Spc(W, B\Aut_{\cC}(C))\simeq \Map_{\Spc}(W, {\bB\Aut_{\cC}(C)})\ .$$
 {In case $\cC=\Spc$, classical fibration theory identifies this $\infty$-groupoid with the $\infty$-groupoid of fibrations over (a topological space representing the space) $W$ with typical fibre $C$, a weak form of a bundle on $W$. }
\hB
 \end{rem}

 Let $\phi$ in $ \Fun(W,{\bB\Aut_{\cC}(C)})$ {be a local system on $W$ with fibre $C$ and assume that $\cC$ admits $W$-indexed limits.}
\begin{definition}\label{wetjigowergferwfwefre}
We define the object in $\cC$ of global sections of the local system by $\Gamma(W,\phi):=\lim_{W} \phi$.
\end{definition}

Coming back  to the bundle of $C^{*}$-algebras $\cA\to X$ {over a compact Hausdorff 
space}
we now explain the construction of  {an associated} 
 {local system $\ell(X) \to \bB\Aut_{\CAlg_h}(A)$.
Recall from \cref{werijgwergwerfwef}} that $\Simp(X)$ denotes the category of singular simplices of $X$. 
The bundle $\cA \to X$ induces a functor $\phi_{\cA}\colon \Simp(X)^{\op}\to \nCalg$ which sends the simplex
$\sigma\colon \Delta^{n}\to X$ to the $C^{*}$- algebra of sections
$\Gamma(\Delta^{n},\sigma^{*}\cA)$. The morphisms between simplices are sent to the corresponding 
restriction maps. Since $\cA$ is locally trivial and the simplices are contractible, classical bundle theory implies that the $C^{*}$-algebras $\Gamma(\Delta^{n},\sigma^{*}\cA)$ are isomorphic to $C(\Delta^n,A)$, and therefore all homotopy equivalent to $A$. Moreover, all the restriction maps along morphisms {of $\Simp(X)$} are homotopy equivalences. The composite
\[ \Simp(X)^\op \stackrel{\phi_\cA}{\lto} \nCalg \stackrel{L_h}{\lto} \nCalg_h \]
therefore sends all morphisms in $\Simp(X)^\op$ to equivalences.
Using the interpretation of $\Simp(X)\to \ell(X)$ as the Dwyer-Kan localization of $\Simp(X)$ at all morphisms
we obtain an   essentially unique functor $\ell(X) \simeq \ell(X)^{\op} \to \nCalg_h$ which by  construction lands in $\bB\Aut_{\nCalg_h}(A)$. This is the local system we intended to construct:

\begin{definition}
We call the resulting map
$$\tilde \phi_{\cA}\colon \ell(X)\to  {\bB\Aut_{\nCalg_{h}}(A)}$$
the local system associated to the bundle $\cA$.
\end{definition}

 {From this associated local system, we can form various other local systems we shall be interested in.}
 {Recall that $\kk \colon \nCalg \to \KK$ factors as $\nCalg \to \nCalg_h \stackrel{\kk_h}{\to} \KK$ and that we have the functor $\K \colon \KK \to \Mod(\KU)$.}
\begin{definition}
We call the composite
\[ \kk(\cA)\colon \ell(X) \stackrel{\tilde{\phi}_\cA}{\lto} \bB\Aut_{\nCalg_h}(A) \stackrel{\kk}{\lto} \bB\Aut_{\KK}(\kk(A)) \]
the $\KK$-local system associated to the bundle $\cA$ and the composite
\[ \K(\cA) \colon \ell(X) \stackrel{\kk(\cA)}{\lto} \bB\Aut_{\KK}(\kk(A)) \stackrel{\K}{\lto} \bB\Aut_{\Mod(\KU)}(K(A)) \]
the $K$-theory local system associated to $\cA$.
\end{definition}
{
We then have the respective global sections
\[\Gamma(X,\kk(\cA)) := \lim_{\ell(X)} \kk(\cA) \:\:\:\mbox{in} \:\:\:\KK \quad \text{ and } \quad \Gamma(X,\K(\cA)) := \lim_{\ell(X)} \K(\cA) \:\:\:\mbox{in} \:\:\: \Mod(\KU).\]
Since $\K \colon \KK \to \Mod(\KU)$ preserves limits (it is corepresented), there is a canonical equivalence
\[ \Gamma(X,\K(\cA)) \simeq \K(\Gamma(X,\kk(\cA)))\]
relating the global sections of the $K$-theory local system of $\cA$ as the K-theory of the global sections of the $\KK$-local system associated to $\cA$.
Furthermore, using that the localization functor $\Simp(X) \to \ell(X)$ is cofinal,
the global sections of the $\KK$-local system can also be expressend as 
\begin{equation}\label{eq:formula-for-limit}
\Gamma(X,\kk(\cA))\simeq \lim_{\Simp(X)^{\op}} \kk_{h}\circ L_h \circ \phi_{\cA}\ .
\end{equation}
}

 \begin{prop} \label{weorgpwegwrewgerg9} 
 If $X$ is homotopy  {finitely dominated}, then we 
 have a canonical equivalence  \begin{equation}\label{svsdcacaescca}
\kk(\Gamma(X,\cA))\simeq \Gamma(X,\kk(A))
\end{equation}
in $\KK$, and consequently an equivalence $K(\Gamma(X,\cA)) \simeq \Gamma(X,\K(\cA))$ in $\Mod(\KU)$.
 \end{prop}
 {The latter describes the $\cA$-twisted $K$-theory of $X$ by homotopy theoretic means, as promised, and the former also gives variants of this result for $\cA$-twisted $KK$-theory.}
\begin{proof}
We consider  the $\KK$-valued functors 
$$(f\colon Y\to X)\mapsto \kk(\Gamma(Y,f^{*}\cA)) \ , \quad (f\colon Y\to X)\mapsto  \Gamma({\ell}(Y),\kk(f^{*}\cA))$$
on compact Hausdorff spaces over $X$. {We aim to show that they are equivalent provided $X$ is homotopy finitely dominated. To this end, we first construct a
natural transformation $\rho$ between these functors and then investigate its behaviour on homotopy finitely dominated spaces}. In order to   describe  the value of $\rho$ on $f\colon Y\to X$ we  consider   the following natural transformation $\rho_f$ of functors
\begin{equation}\label{verwvijerovewvsdfvsfv}\underline{\Gamma(Y,f^{*}\cA)}\to \phi_{f^{*}\cA}\colon \Simp(Y)^{\op}\to \nCalg
\end{equation}  where the underline denotes the constant functor. The transformation $\rho_f$ on the simplex $\sigma\colon \Delta^{n}\to Y$ is given by the restriction
$$\sigma^{*}\colon \Gamma(Y,f^{*}\cA)\to \Gamma(\Delta^{n},\sigma^{*}f^{*}\cA)\ .$$
Composing with the canonical functor $\kk \colon \nCalg \to \KK$, we get a transformation
$$\underline{\kk(\Gamma(Y,f^{*}\cA))}\to \kk_{h}\circ L_{h}\circ \phi_{f^{*}\cA}\colon \Simp(Y)^{\op}\to \KK\ ,$$
or equivalently, by
 adjunction and the equivalence \eqref{eq:formula-for-limit},  a map $\rho(f)$
 \begin{equation}\label{qewfeqwdqed}
\kk(\Gamma(Y,f^{*}\cA))\to \lim_{\Simp(Y)^{\op}} \kk_{h}\circ L_{h}\circ \phi_{f^{*}\cA} \simeq \Gamma(\ell(Y),\kk(f^{*}\cA))\ .
\end{equation}
In order  construct $\rho$ as a natural transformation,
we consider the category $\CH_{/X}$ of compact Hausdorff spaces $Y\to X$ over $X$  {and let $\widetilde \CH_{/X}$ be the Grothendieck construction of the functor
$\CH_{/X}\to \CH \to \Cat$ which sends $Y\to X$ to $\Simp(Y)^{\op}$. That is, $\widetilde \CH_{/X}$ is the category of pairs $(f ,\sigma )$ of an object $f\colon Y\to X$ in $\CH_{/X}$ and
an object $\sigma\colon \Delta^{n}\to Y$ in $\Simp(Y)^{\op}$. A morphism $(f,\sigma)\to (f',\sigma')$ consists of a morphism
$\psi\colon f\to f'$ in $\CH_{/X}$ and a morphism $\psi\circ \sigma\to \sigma'$ in $\Simp(Y')^{\op}$.
We let $q\colon \widetilde \CH_{/X}\to \CH_{/X}$ denote the forgetful functor.}

We have a functor
$\Gamma_{\cA}\colon \CH_{/X}\to \nCalg$ sending $f:Y\to X$ to $\Gamma(Y,f^{*}\cA)$.
We further consider the functor $\phi_{\cA}\colon \widetilde \CH_{/X}\to \nCalg$ which sends
$(f ,\sigma )$ to $\Gamma(\Delta^{n},\sigma^{*}f^{*}\cA)$.
The analogue of construction \eqref{verwvijerovewvsdfvsfv} in this context gives a natural transformation
$$q^{*}\Gamma_{\cA}\to \phi_{\cA}\colon \widetilde \CH_{/X}\to \nCalg \ $$
which in turn induces a transformation 
$$q^{*} (\kk\circ \Gamma_{\cA})\to \kk_{h}\circ L_{h} \circ \phi_{\cA}\colon \widetilde \CH_{/X}\to \KK\ ,$$
and therefore by adjunction
$$\rho: \kk\circ \Gamma_{\cA}\to q_{*}( \kk_{h}\circ L_{h} \circ \phi_{\cA})\colon \CH_{/X}\to \KK\ ,$$
where $q_{*}$ is the right Kan extension functor. Using the point-wise formula for the right Kan extension one sees that the target of this transformation is the correct functorial enhancement of the construction $(f\colon Y\to X)\mapsto \Gamma(\ell(Y),\kk(f^{*}\cA))$. \hB

We now observe that both sides of the natural transformation  $\rho$
are homotopy invariant on the category of compact Hausdorff spaces over $X$.  {Moreover, since retracts of equivalences are equivalences, it suffices to treat the case where}
$X$ is a finite CW-complex.  {In this case, we want to show that
the value of the natural transformation constructed above on $X$ (or rather on the identity of $X$) is an equivalence.}
We will  argue  by induction over the cells of $X$.
Assume that 
 $Z,Y$ are subcomplexes of $X$ and $Y$ is obtained from $Z$ by attaching
an $n$-cell $D^{n}$. {The transformation just constructed induces a map from the left to the right between the following squares.}
\[ \begin{tikzcd}
	\kk(\Gamma(Y,\cA_{|Y}) \ar[r] \ar[d] & \kk(\Gamma(Z,\cA_{|Z})) \ar[d] & \Gamma(\ell(Y),\kk(\cA_{|Y})) \ar[r] \ar[d] & \Gamma(\ell(Z),\kk(\cA_{|Z})) \ar[d] \\\hspace{-1cm}
	\kk(\Gamma(D^n,\cA_{|D^n})) \ar[r] & \kk(\Gamma(\partial D^n,\cA_{|\partial D^n})) & \Gamma(\ell(D^n),\kk(\cA_{|D^n})) \ar[r] & \Gamma(\ell(\partial D^n),\kk(\cA_{|\partial D^n}))
\end{tikzcd}\]
 {Now, both of these squares are pullback squares. By the inductive assumption, the map on the top right corner is an equivalence. Moreover, the maps on the bottom two terms identify with the equivalences of \cref{wergijowegwergwrf} since the bundle $\cA$ restricted to $D^n$ and $\partial D^n$ is trivial:
 Indeed, in the case of a trivial bundle $\cA$, the transformation we have constructed induces a natural transformation between functors that depend only on the underlying space, and hence this transformation is determined by its value on the terminal space, see \cref{wrgwtrgwrgfwerfwerf}, in which case it is the identity by construction. We deduce that the map on the top left corners is also an equivalence, providing the inductive step.}
\end{proof}

\bibliographystyle{plain}
\bibliography{survey}

\end{document}